\RequirePackage{fix-cm}
\documentclass{siamart171218}
\newsiamremark{remark}{Remark}
\usepackage[caption=false]{subfig}
\usepackage{amsmath}
\usepackage{amssymb}
\usepackage{mathtools}
\usepackage{breqn}
\usepackage{stmaryrd}

\usepackage{pgfplots, pgfplotstable, booktabs, colortbl, array}
\usepackage{tikz}
\pgfplotsset{compat=newest}
\usepackage{pgffor}
\usepackage{xcolor}

\DeclareMathOperator{\dv}{div}
\DeclareMathOperator{\Tr}{Tr}
\DeclareMathOperator{\supp}{supp}
\DeclareMathOperator{\Hess}{Hess}

\begin{document}

\title{High-Order Approximation of Gaussian Curvature with Regge Finite Elements\thanks{The author was supported in part by the NSF under grant DMS-1703719.}}

\author{Evan S. Gawlik\thanks{Department of Mathematics, University of Hawaii at Manoa
    (\email{egawlik@hawaii.edu})}}

\date{}

\headers{Gaussian Curvature with Regge Finite Elements}{E. S. Gawlik}

\maketitle

\begin{abstract}
A widely used approximation of the Gaussian curvature on a triangulated surface is the angle defect, which measures the deviation between $2\pi$ and the sum of the angles between neighboring edges emanating from a common vertex.  We show that the linearization of the angle defect about an arbitrary piecewise constant Regge metric is related to the classical Hellan-Herrmann-Johnson finite element discretization of the $\dv \dv$ operator.  Integrating this relation leads to an integral formula for the angle defect which is well-suited for analysis and generalizes naturally to higher order.  We prove error estimates for these high-order approximations of the Gaussian curvature in $H^k$-Sobolev norms of integer order $k \ge -1$.
\end{abstract}


\section{Introduction} \label{sec:intro}

One of the most widely used approximations of the Gaussian curvature on a triangulated surface is the angle defect: $2\pi$ minus the sum of the angles between neighboring edges emanating from a common vertex. This approximation (and its generalization to higher dimensions) is used in several applications, including discrete analogues of Ricci flow~\cite{chow2003combinatorial,jin2008discrete}, discrete theories of relativity~\cite{regge1961general,alsing2011simplicial}, discrete differential geometry~\cite{glickenstein2011discrete,sullivan2008curvatures}, and computer graphics algorithms~\cite{kharevych2006discrete,crane2010trivial,springborn2008conformal}.  
Despite its widespread use, the angle defect leaves much to be desired if one is interested in accurately approximating the curvature of a smooth surface (or smooth Riemannian manifold) with a discretization thereof.  It is manifestly a low-order approximation of the curvature, relying in essence on piecewise constant approximations of the underlying smooth metric tensor.   

In this paper, we introduce and analyze a family of high-order approximations of the Gaussian curvature using piecewise polynomial approximations of the metric tensor.  The cornerstone of our construction is an integral formula for the angle defect that mimics a certain integral formula for the Gaussian curvature which is valid in the smooth setting.  In the discrete setting, the integral formula follows from the observation that the linearization of the angle defect about an arbitrary piecewise constant metric (more precisely, a piecewise constant Regge metric) is related to the classical Hellan-Herrmann-Johnson finite element discretization of the $\dv \dv$ operator.  This observation generalizes one made by Christiansen~\cite{christiansen2011linearization}, who derived the linearization of the angle defect about the Euclidean metric and related it to the jumps in the tangential-normal components of the metric perturbation (see Remark~\ref{remark:christiansen} for more details). 

To generalize the angle defect to higher order, we rely on the Regge finite element spaces recently developed by Li~\cite{li2018regge}, which have their origins in the work of Christiansen~\cite{christiansen2011linearization}.  These finite element spaces consist of piecewise polynomial $(0,2)$-tensor fields with continuous tangential-tangential components across element interfaces.  In the lowest order setting, a positive definite Regge finite element realizes a piecewise flat triangulation whose (squared) edge lengths correspond to the degrees of freedom for the finite element space.


\begin{figure}
\centering
\subfloat[]{\label{fig:kappa:a}\includegraphics[scale=0.18,trim=1.6in 0.6in 1.3in 0.4in,clip=true]{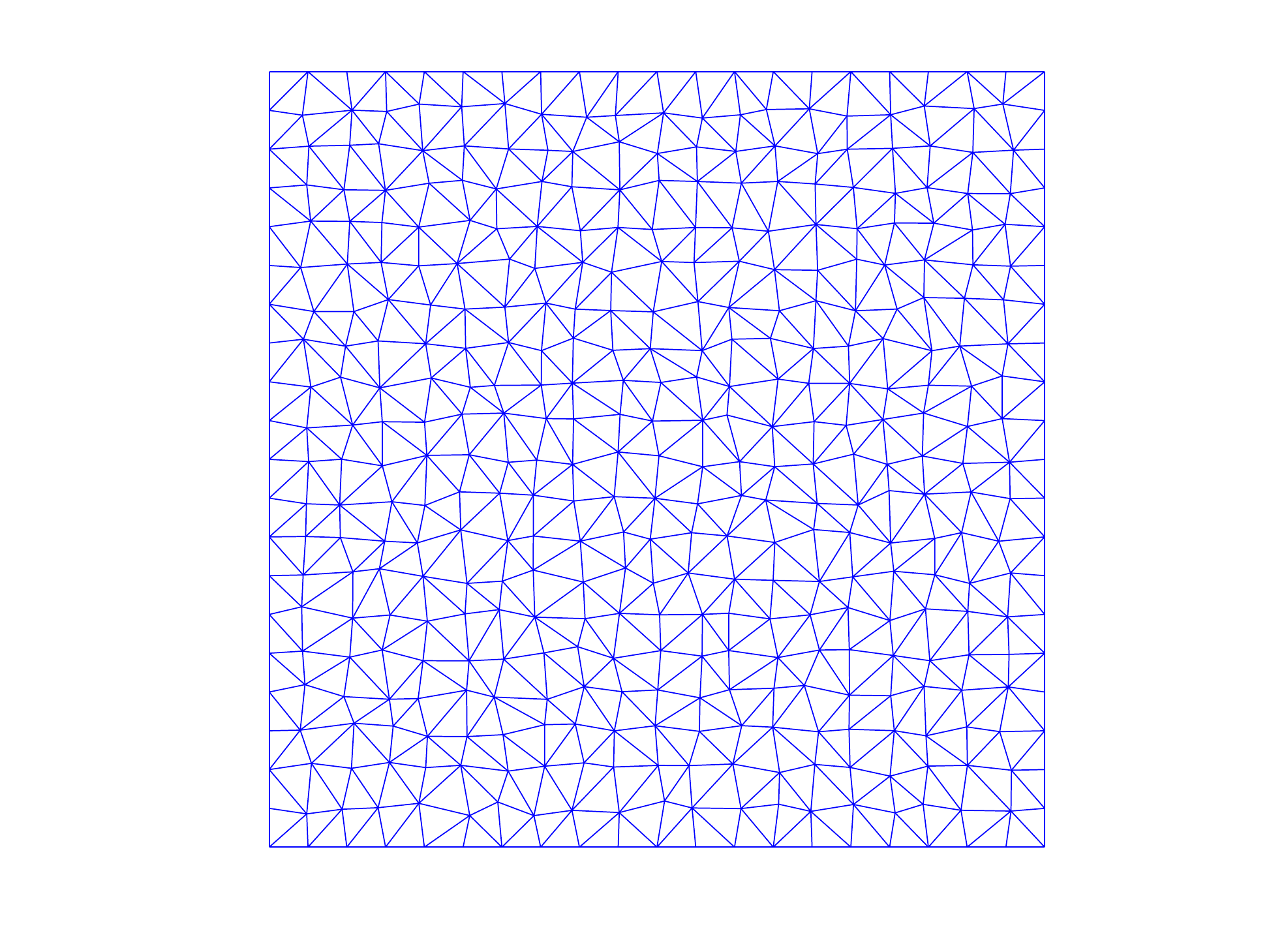}} \hspace{0.05in}
\subfloat[$r=0$]{\label{fig:kappa:b}\includegraphics[scale=0.22,trim=1.0in 0.9in 0.7in 1.5in,clip=true]{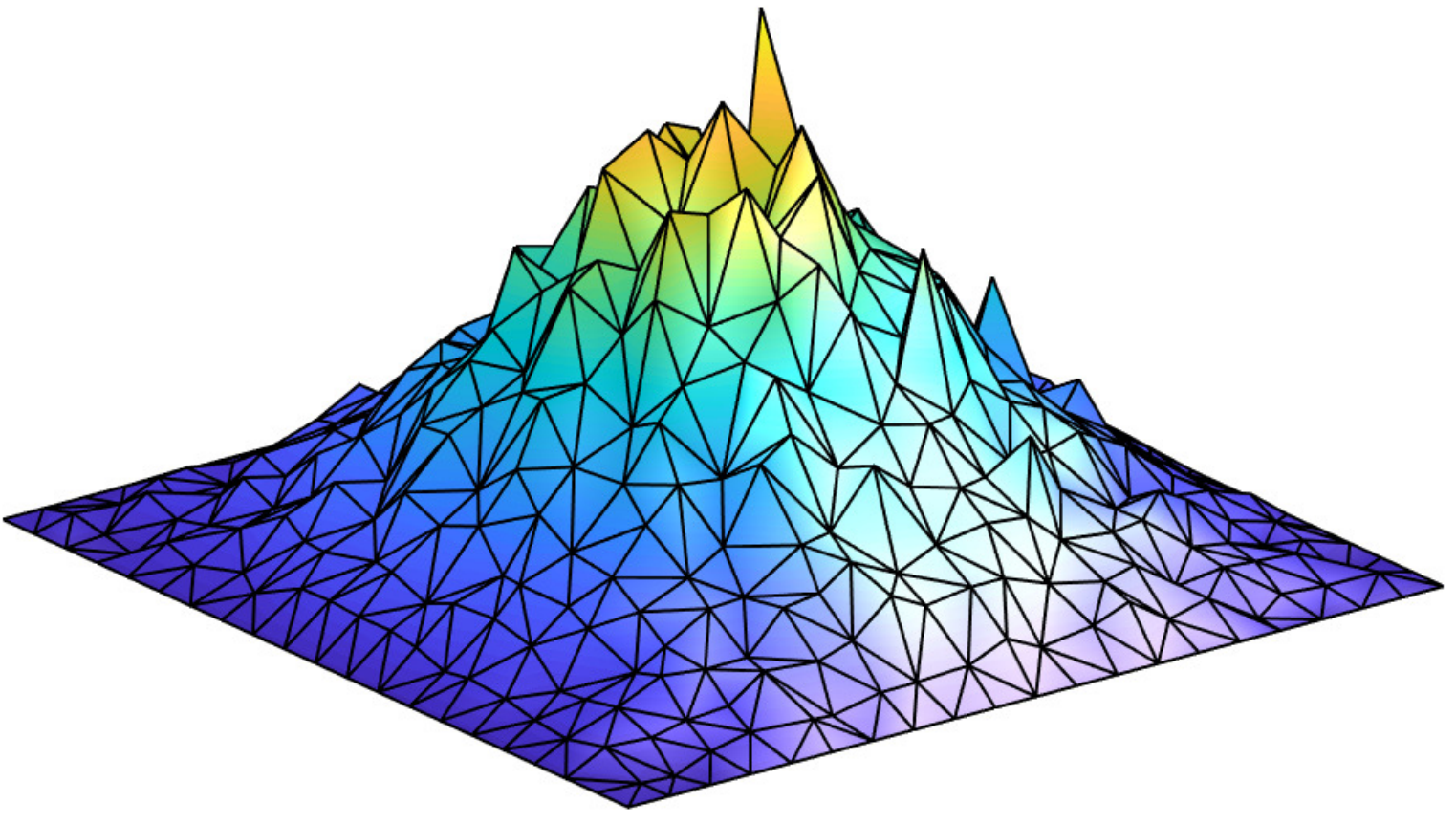}} 
\subfloat[$r=1$]{\label{fig:kappa:c}\includegraphics[scale=0.22,trim=1.0in 0.9in 0.7in 1.3in,clip=true]{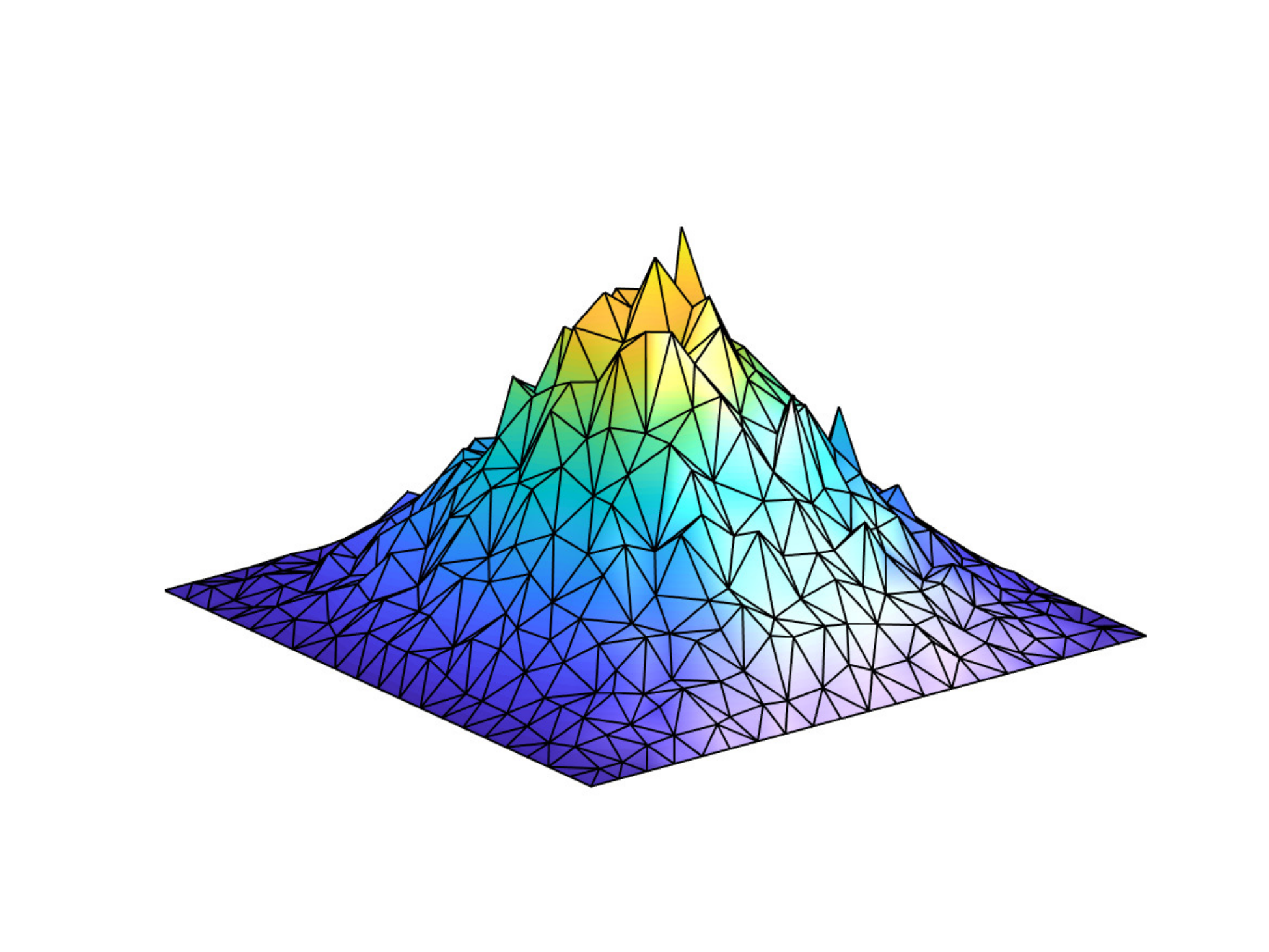}}
\subfloat[$r=2$]{\label{fig:kappa:d}\includegraphics[scale=0.22,trim=1.0in 0.5in 0.7in 1.5in,clip=true]{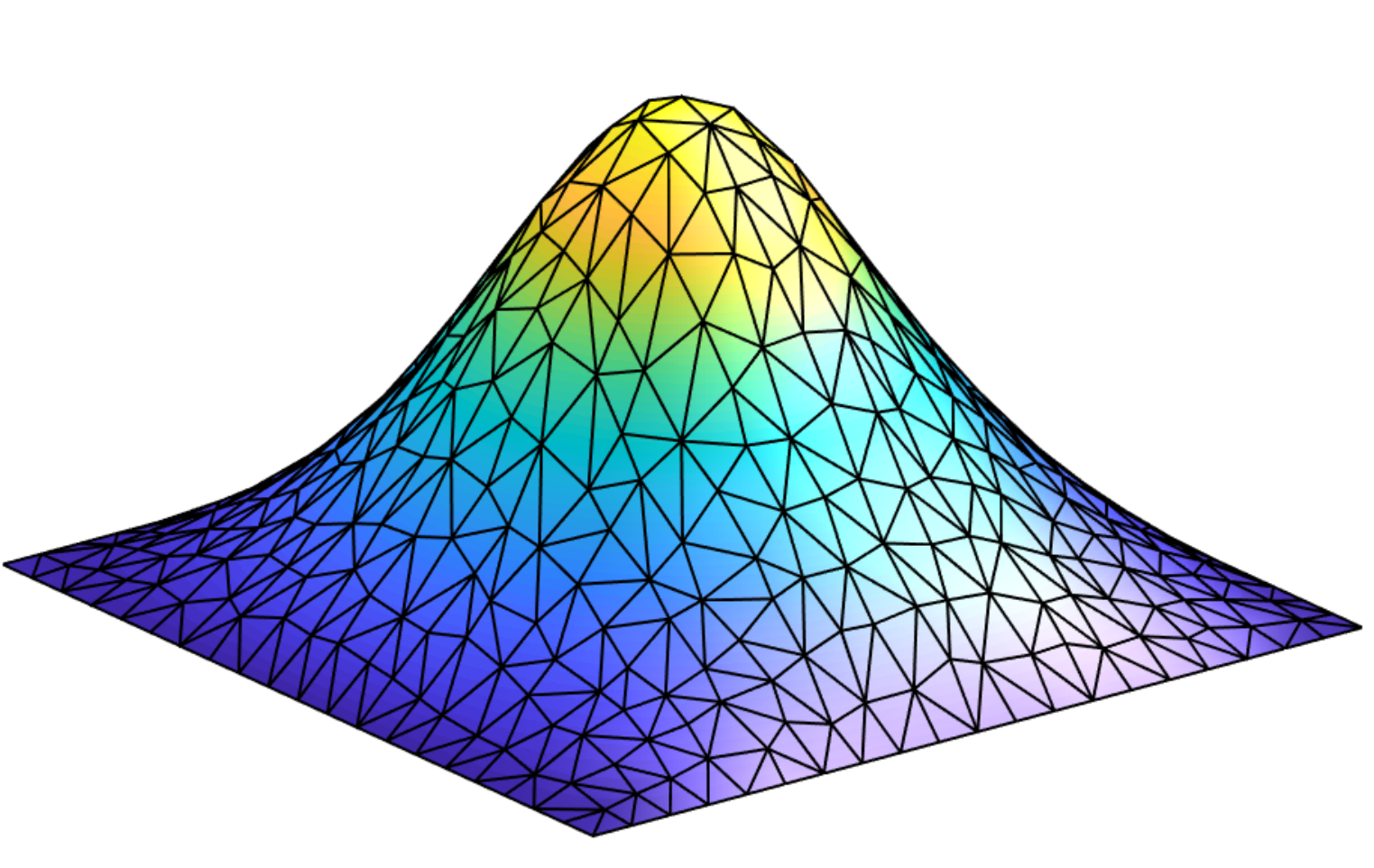}}
\caption{Approximate Gaussian curvature of the metric~(\ref{gexact}), computed using piecewise polynomial approximations of $g$ of degree $r=0,1,2$ on the triangulation of $(-1,1) \times (-1,1)$ depicted in (a).}
\label{fig:kappa}
\end{figure}

The advantages of high-order approximation of the Gaussian curvature are easily illustrated with an example.  Consider the square $\Omega = (-1,1) \times (-1,1)$ equipped with the Riemannian metric
\begin{equation} \label{gexact}
g(x,y) = \begin{pmatrix} 1 + \left( \frac{\partial f}{\partial x} \right)^2 & \frac{\partial f}{\partial x} \frac{\partial f}{\partial y} \\ \frac{\partial f}{\partial x} \frac{\partial f}{\partial y} & 1 + \left( \frac{\partial f}{\partial y} \right)^2 \end{pmatrix},
\end{equation}
where $f(x,y) = \frac{1}{2}x^2 - \frac{1}{12}x^4 +  \frac{1}{2}y^2 - \frac{1}{12}y^4$.
This is nothing more than the induced metric for the surface $z = f(x,y)$ in $\mathbb{R}^3$.  The exact Gaussian curvature of $g$ is
\[
\kappa(g)(x,y) = \frac{81 (1-x^2) (1-y^2) }{ \left( 9 + x^2 (x^2-3)^2 + y^2 (y^2-3)^2 \right)^2}.
\]
Figure~\ref{fig:kappa} plots the approximate Gaussian curvature of $g$, as computed using Definition~\ref{def:kappah} below with piecewise polynomial approximations of $g$ of degree $r=0,1,2$ on the triangulation of $\Omega$ depicted in Figure~\ref{fig:kappa:a}.  As we will show later in this paper, the approximate curvature produced by Definition~\ref{def:kappah} in the case $r=0$ is precisely the angle defect, normalized by the consistent mass matrix for piecewise linear finite elements.  
(In the interest of fairness, we normalized by a lumped mass matrix to produce Figure~\ref{fig:kappa:b}; the appearance of Figure~\ref{fig:kappa:b} worsens if the consistent mass matrix is used.)
Notice that the results for $r=0$ and $r=1$ are not particularly satisfactory.
This should come as no surprise; one expects the second derivatives of a degree-$r$ polynomial approximation of $g$ to converge in $W^{r-2,p}$-Sobolev norms under mesh refinement, but not in stronger norms.  A chief goal of this paper is to verify this intuition with an error estimate.
 
Our error analysis complements a number of related results in the literature on scalar curvature approximation.  
Cheeger, M{\"u}ller, and Schrader~\cite{cheeger1984curvature} prove that the angle defect converges in the sense of measures to the (densitized) scalar curvature if a smooth Riemannian manifold (not necessarily of dimension 2) is approximated with a suitable sequence of triangulations.  Christiansen~\cite{christiansen2013exact} proves a dual result: if a piecewise constant metric is approximated with a sequence of mollifications thereof, then the exact (densitized) scalar curvature of the mollified metric converges in the sense of measures to the angle defect.  Other analyses of the angle defect appear in~\cite{borrelli2003angular,xu2006convergence} and the references therein; many of these analyses are guided by Taylor expansions on parametrized surfaces and impose special conditions on the triangulation.  Our error analysis will focus not on the angle defect but instead on its higher order generalizations.

There appear to be relatively few studies on Gaussian curvature approximation that offer quantitative error bounds in Sobolev norms like the bounds in Theorem~\ref{thm:err} below.  Fritz~\cite{fritz2013isoparametric} has proven bounds of this type for a curvature approximation which uses isoparametric approximations of surfaces embedded in $\mathbb{R}^3$.  His results apply more generally to Ricci tensor approximation on hypersurfaces, and have been used to discretize Ricci flow~\cite{fritz2015numerical}.  However, they are inapplicable if the manifold does not admit a codimension-one embedding into Euclidean space.  Note that studies related to \emph{mean} curvature approximation are more widespread; see, for instance,~\cite{hildebrandt2006convergence,bansch2004surface,dziuk2006finite,kovacs2018convergent}.

This paper is organized as follows.  In Section~\ref{sec:smooth}, we introduce our notation and point out an integral formula for the Gaussian curvature.  In Section~\ref{sec:discretization}, we discretize this integral formula with the aid of Regge finite elements, and we show that it reduces to the angle defect in the lowest-order setting.  We state and prove error estimates for the aforementioned discretization in Section~\ref{sec:error}.  We conclude with numerical examples in Section~\ref{sec:numerical}.

For simplicity, we perform much of the forthcoming analysis on a triangulated polygonal domain in $\mathbb{R}^2$ equipped with a non-Euclidean metric.  Note, however, that the curvature approximations we introduce are coordinate-free and can be readily applied to two-dimensional orientable simplicial complexes with more general topology.  We refer the reader to~\cite{li2018regge} for the appropriate definitions of the Regge finite elements in this general setting.  Importantly, our curvature approximations do not rely on (nor assume the existence of) an embedding of the manifold under consideration in $\mathbb{R}^3$.

\section{Curvature in the Smooth Setting} \label{sec:smooth}

\subsection{Preliminaries} \label{sec:prelim}

Let $\Omega \subset \mathbb{R}^2$ be a polygonal domain.  We use standard notation for the Sobolev spaces $W^{s,p}(\Omega)$ of differentiability index $s \in \mathbb{R}$ and integrability index $p \in [1,\infty]$.  We denote $H^s(\Omega) = W^{s,2}(\Omega)$ and $H^1_0(\Omega) = \{ v \in H^1(\Omega) \mid \left. v\right|_{\partial\Omega} = 0 \}$.  

Let $g$ be a smooth Riemannian metric on $\Omega$.  Fix coordinates $(x^1,x^2)$ on $\mathbb{R}^2$ so that $g$ may be regarded as a map from $\Omega$ to $\mathbb{S} = \{ \sigma \in \mathbb{R}^{2 \times 2} \mid \sigma = \sigma^T\}$.  If $v$ is a scalar field on $\Omega$, then we denote $\nabla v = \left( \frac{\partial v}{\partial x^1}, \frac{\partial v}{\partial x^2} \right)^T$ and $\nabla_g v = g^{-1} \nabla v$.  If $w$ is a vector field on $\Omega$, then we regard it as a column vector and denote $\dv w = \frac{\partial w^1}{\partial x^1} +  \frac{\partial w^2}{\partial x^2}$ and $\dv_g w = \frac{1}{\sqrt{\det g}} \dv(\sqrt{\det g}\, w)$.  

The Riemannian Hessian of a scalar field $v$ is denoted $\Hess_g v$.  In coordinates,
\[
(\Hess_g v)_{ij} =  \frac{\partial^2 v}{\partial x^i \partial x^j} - \Gamma^k_{ij} \frac{\partial v}{\partial x^k}.
\]
Here, the Einstein summation convention is adopted, and $\Gamma^k_{ij}$ are the Christoffel symbols of the second kind.  That is,
\[
\Gamma^k_{ij} = \frac{1}{2} g^{k\ell} \left( \frac{\partial g_{\ell i}}{\partial x^j} + \frac{\partial g_{\ell j}}{\partial x^i} - \frac{\partial g_{ij}}{\partial x^\ell}  \right),
\]
where $g_{ij}$ denotes the $(i,j)$-component of $g$, and $g^{ij}$ denotes the $(i,j)$-component of $g^{-1}$.  
The Laplacian of $v$ is 
\[
\Delta_g v = \dv_g \nabla_g v.
\]
If $\sigma$ is a symmetric $(0,2)$-tensor field on $\Omega$, then we regard it as a map from $\Omega$ to $\mathbb{S}$ and denote its components by $\sigma_{ij}$.  
In a slight abuse of notation, we denote by $\dv_g \sigma$ the vector field with components
\[
(\dv_g \sigma)^i = \frac{\partial \sigma^{ij}}{\partial x^j} + \Gamma^i_{jk} \sigma^{kj} + \Gamma^j_{jk}\sigma^{ik},
\]
where $\sigma^{ij} = g^{ik}\sigma_{k\ell}g^{\ell j}$.  We omit the subscript $g$ if $g$ is the Euclidean metric $\delta \equiv \left( \begin{smallmatrix} 1 & 0 \\ 0 & 1 \end{smallmatrix} \right)$.  Thus, $\dv \sigma = \dv_\delta \sigma$, $\Delta v = \Delta_\delta v$, etc.

The Gaussian curvature of $g$ is denoted $\kappa(g)$; it is half the scalar curvature $R(g)$:
\[
\kappa(g) = \frac{1}{2}R(g) = \frac{1}{2} g^{ij} \left( \frac{\partial \Gamma^k_{ij}}{\partial x^k} - \frac{\partial \Gamma^k_{ik}}{\partial x^j} +\Gamma^\ell_{ij} \Gamma^k_{k\ell} - \Gamma^\ell_{ik} \Gamma^k_{j\ell} \right).
\]
Because $\Omega$ is two-dimensional, the Ricci tensor $\operatorname{Ric}(g)$ is proportional to the metric:
\begin{equation} \label{Ric}
\operatorname{Ric}(g) = \kappa(g) g.
\end{equation}

\subsection{Linearization of the Curvature} \label{sec:lin}
Let $\mu(g) = \sqrt{\det g} \, dx$ be the volume form on $\Omega$ induced by $g$.  The linearization of $(\kappa \mu)(g) = \kappa(g)\mu(g)$ about $g$ will be of fundamental importance in this paper.  For any smooth $\sigma : \Omega \rightarrow \mathbb{S}$, we have
\begin{equation} \label{lin}
D(\kappa\mu)(g) \cdot \sigma = \frac{1}{2} (\dv_g \dv_g S_g \sigma) \mu(g),
\end{equation}
where
\begin{equation} \label{Sg}
S_g \sigma = \sigma - g \Tr(g^{-1}\sigma).
\end{equation}
This can be seen by combining the well-known relations~(\cite[Lemma 2]{fischer1975deformations},~\cite[Equation 2.11]{chow2006hamilton})
\begin{align*}
D\kappa(g) \cdot \sigma &= \frac{1}{2} \left( \dv_g \dv_g \sigma - \Delta_g \Tr(g^{-1} \sigma) - \Tr(g^{-1}\sigma g^{-1} \operatorname{Ric}(g)) \right),  \\
D\mu(g) \cdot \sigma &= \frac{1}{2} \Tr(g^{-1}\sigma) \mu(g),
\end{align*}
with~(\ref{Ric}), noting that $\dv_g \dv_g (gv) = \Delta_g v$ for any scalar field $v$.

Another way of writing~(\ref{lin}) is as follows.  Let 
\[
\langle u, v \rangle_g = \int_\Omega u v \, \mu(g)
\]
be the $L^2$-inner product on $\Omega$ induced by $g$.  Then, for any scalar function $v$,
\begin{equation} \label{linint}
\left.\frac{d}{dt}\right|_{t=0} \langle \kappa(g+t\sigma), v \rangle_{g+t\sigma} = \frac{1}{2} \langle \dv_g \dv_g S_g \sigma, v \rangle_g.
\end{equation}
Since the Euclidean metric $\delta$ has zero curvature, integrating the above relation leads to the integral formula
\begin{equation} \label{int}
\langle \kappa(g), v \rangle_g = \frac{1}{2} \int_0^1 b((1-t)\delta+tg; g-\delta, v) \, dt,
\end{equation}
where
\[
b(g;\sigma,v) = \langle \dv_g \dv_g S_g \sigma, v \rangle_g.
\]
Our strategy for discretizing $\kappa(g)$ will be to discretize the integral formula~(\ref{int}).

\section{Discretization} \label{sec:discretization}

Let $\{\mathcal{T}_h\}_{h>0}$ be a shape-regular, quasi-uniform family of triangulations of $\Omega$ parametrized by $h = \max_{K \in \mathcal{T}_h} h_K$, where $h_K = \operatorname{diam}(K)$ denotes the diameter of a triangle $K$.  In other words, there are constants $C_1$ and $C_2$ such that for every $h$,
\begin{align}
\max_{K \in \mathcal{T}_h} \frac{h_K}{\rho_K} &\le C_1, \\
\max_{K \in \mathcal{T}_h} \frac{h}{h_K} &\le C_2, 
\end{align}
where $\rho_K$ denotes the inradius of $K$.

Let $e \subset \partial K$ be an edge of a triangle $K \in \mathcal{T}_h$.  The outward unit normal vector to $K$ along $e$ relative to the Euclidean metric is denoted $n$, and the unit tangent vector relative to the Euclidean metric is $\tau = -Jn$, where $J = \left( \begin{smallmatrix} 0 & 1 \\ -1 & 0 \end{smallmatrix} \right)$.  Relative to $g$, the unit tangent and normal vectors are
\[
\tau_g = \frac{1}{\sqrt{\tau^T g \tau}} \tau, \quad n_g = \frac{J g \tau}{\sqrt{\tau^T g \tau}{\sqrt{\det g}}}. 
\]
One checks that $\tau_g^T g \tau_g = n_g^T g n_g = 1$ and $\tau_g^T g n_g = 0$ since $gJg = (\det g)J$.  We also note that if $\sigma : \Omega \rightarrow \mathbb{S}$, then
\begin{equation} \label{nSn}
n_g^T (S_g \sigma) n_g = -\tau_g^T \sigma \tau_g,
\end{equation}
owing to the definition~(\ref{Sg}) of $S_g$ and the identity $\tau_g \tau_g^T + n_g n_g^T = g^{-1}$.

Let $\mathcal{E}_h$ denote the set of edges of triangles in $\mathcal{T}_h$, and let $\mathring{\mathcal{E}}_h$ denote the set of interior edges; these are the edges $e \in \mathcal{E}_h$ with $e \not\subset \partial \Omega$.  Let $v$ be a scalar field.  Along any edge $e = K_1 \cap K_2 \in \mathring{\mathcal{E}}_h$, we denote 
\[
\left\llbracket v \right\rrbracket = \left. v \right|_{\partial K_1} + \left. v \right|_{\partial K_2}.
\]
If $e \in \mathcal{E}_h$ is on the boundary of $\Omega$, we denote
\[
\left\llbracket v \right\rrbracket = \left.v\right|_{\partial\Omega}.
\]

Let
\[
V = \{ v \in H^1_0(\Omega) \mid \left.v\right|_K \in H^2(K), \, \forall K \in \mathcal{T}_h \}
\]
and
\begin{equation*}
\begin{split}
\Sigma = \{ \sigma \in L^2(\Omega) \otimes \mathbb{S} \mid \left.\sigma\right|_K \in H^1(K) \otimes \mathbb{S}, \, \forall K \in \mathcal{T}_h, \text{ and } \\ \tau^T \sigma \tau \text{ is continuous across every } e \in \mathring{\mathcal{E}}_h \}.
\end{split}
\end{equation*}
Note that $V$ and $\Sigma$ depend on $h$, but we have omitted a subscript $h$ to emphasize that they are infinite-dimensional spaces.  We define a metric-dependent bilinear form $b_h(g; \cdot, \cdot) : \Sigma \times V \rightarrow \mathbb{R}$ by
\begin{align*}
b_h(g;\sigma,v) 
&= \sum_{K \in \mathcal{T}_h} \langle S_g \sigma, \Hess_g v \rangle_{g,K} + \sum_{e \in \mathcal{E}_h} \left\langle \tau_g^T \sigma \tau_g, \left\llbracket \frac{\partial v}{\partial n_g} \right\rrbracket \right\rangle_{g,e},
\end{align*}
where
\begin{align*}
\langle u, v \rangle_{g,e} &= \int_e u v  \sqrt{\tau^T g \tau} \, d\ell, \\
\langle \sigma, \rho \rangle_{g,K} &= \int_K \Tr(g^{-1}\sigma g^{-1} \rho) \, \mu(g),
\end{align*}
and
\[
\frac{\partial v}{\partial n_g} = n_g ^T g \nabla_g v.
\]
The bilinear form $b_h$ is well-studied.  In view of~(\ref{nSn}), it is (up to the appearance of $S_g$) a non-Euclidean generalization of the bilinear form used to discretize the $\dv \dv$ operator in the classical Hellan-Herrmann-Johnson mixed discretization of the biharmonic equation~\cite{babuska1980analysis,arnold1985mixed,brezzi1977mixed}.  See~\cite[Section 4.2]{li2018regge} for more insight into this connection in the Euclidean setting.

Now let $q \in \mathbb{N}$ and $r \in \mathbb{N}_0$, and define finite-dimensional subspaces
\[
V_h = \{v \in V \mid \left. v \right|_K \in \mathcal{P}_q(K), \, \forall K \in \mathcal{T}_h \}
\]
and
\[
\Sigma_h = \{\sigma \in \Sigma \mid \left. \sigma \right|_K \in \mathcal{P}_r(K) \otimes \mathbb{S}, \, \forall K \in \mathcal{T}_h \},
\]
where $\mathcal{P}_r(K)$ denotes the space of polynomials of degree $\le r$ on $K$.
The space $\Sigma_h$ is the space of \emph{Regge finite elements} of degree $r$~\cite{li2018regge,christiansen2011linearization}.

Of particular importance to us will be the space of positive definite Regge finite elements,
\[
\Sigma_{h+} = \{ \sigma \in \Sigma_h \mid \sigma(x) \succ 0, \, \forall x \in \mathring{K}, \, \forall K \in \mathcal{T}_h \}.
\]
We define the discrete Gaussian curvature of a metric $g_h \in \Sigma_{h+}$ as follows.

\begin{definition} \label{def:kappah}
Let $q \in \mathbb{N}$, $r \in \mathbb{N}_0$, and $g_h \in \Sigma_{h+}$.  The \emph{discrete Gaussian curvature} of $g_h$ is the unique function $\kappa_h(g_h) \in V_h$ satisfying
\begin{equation} \label{integralh}
\langle \kappa_h(g_h), v_h \rangle_{g_h} = \frac{1}{2} \int_0^1 b_h( (1-t)\delta + t g_h; g_h - \delta, v_h) \, dt, \quad \forall v_h \in V_h.
\end{equation}
\end{definition}

Note that for each $v_h \in V_h$, the value of $\langle \kappa_h(g_h), v_h \rangle_{g_h}$ depends only on the values of $g_h$ in $\supp(v_h)$.  Thus, Definition~\ref{def:kappah} extends readily to orientable triangulations with more general topology: for each function $v_h$ in the canonical basis for $V_h$, one computes the spatial integral in~(\ref{integralh}) over $\supp(v_h)$, which (generically) is a patch of triangles admitting a Euclidean metric.
 
One of our main reasons for favoring this definition is that when $(r,q)=(0,1)$, $\kappa_h(g_h)$ reduces to the popular angle defect.  
In detail, let $\{y^{(i)}\}_{i=1}^N \subset \Omega$ be the vertices of $\mathcal{T}_h$, enumerated in such a way that $y^{(i)} \notin \partial \Omega$ if and only if $1 \le i \le N_0<N$.   Let $\{\phi_i\}_{i=1}^{N_0}$ be the basis for $V_h$ defined by
\[
\phi_i(y^{(j)}) = \begin{cases} 1, &\mbox{ if } i=j, \\ 0, &\mbox{ if } i \neq j, \end{cases} \quad i=1,2,\dots,N_0, \, j=1,2,\dots,N.
\]

\begin{theorem}
Let $r=0$, $q=1$, and $g_h \in \Sigma_{h+}$.  For every $i=1,2,\dots,N_0$, we have
\[
\langle \kappa_h(g_h), \phi_i \rangle_{g_h} = 2\pi - \sum_{K \in \omega_i} \theta_{iK},
\]
where $\omega_i$ is the set of triangles in $\mathcal{T}_h$ sharing vertex $i$, and $\theta_{iK}$ is the interior angle of $K$ at vertex $i$ as measured by $g_h$.
\end{theorem}

The preceding theorem is a consequence of the following identity that mimics~(\ref{linint}).

\begin{lemma} \label{thm:dkappah}
Let $r=0$, $q=1$, $g_h \in \Sigma_{h+}$, $\sigma_h \in \Sigma_h$, and $i \in \{1,2,\dots,N_0\}$.  For each $K \in \omega_i$ and $t$ sufficiently small, let $\theta_{iK}(t)$ be the interior angle of $K$ at vertex $i$ as measured by $g_h + t \sigma_h$.
Then
\begin{equation} \label{dkappah}
\left.\frac{d}{dt}\right|_{t=0} \left( 2\pi - \sum_{K \in \omega_i} \theta_{iK}(t) \right) = \frac{1}{2} b_h(g_h; \sigma_h, \phi_i).
\end{equation}
\end{lemma}
\begin{proof}
Let $K \in \omega_i$ be a triangle with edges $e_a$, $e_b$, and $e_c$ of length $a(t)$, $b(t)$, and $c(t)$ relative to $g_h+t\sigma_h$.  Assume that $e_c$ is opposite the angle $\theta_{iK}(t)$.  Differentiating the law of cosines
\[
a(t)^2+b(t)^2-c(t)^2 = 2a(t)b(t)\cos\theta_{iK}(t)
\]
with respect to $t$ at $t=0$ and solving for $\dot{\theta}_{iK} = \dot{\theta}_{iK}(0)$ gives
\[
\dot{\theta}_{iK} = \frac{-\dot{a} (a-b\cos\theta_{iK}) - \dot{b} (b-a\cos\theta_{iK}) + \dot{c} c}{2A},
\]
where $A = \frac{1}{2}ab\sin\theta_{iK}$ is the area of $K$ relative to $g_h$ and $a=a(0)$, $\dot{a}=\dot{a}(0)$, etc.  
On the other hand, we have
\begin{equation} \label{dvdnabc}
\left.\frac{\partial \phi_i}{\partial n_{g_h}}\right|_{e_a} =\frac{a-b\cos\theta_{iK}}{2A}, \quad \left.\frac{\partial \phi_i}{\partial n_{g_h}}\right|_{e_b} =\frac{b-a\cos\theta_{iK}}{2A}, \quad \left.\frac{\partial \phi_i}{\partial n_{g_h}}\right|_{e_c} = -\frac{c}{2A},
\end{equation}
and 
\begin{equation} \label{intabc}
\langle \tau_{g_h}^T \sigma_h \tau_{g_h}, 1 \rangle_{g_h,e_a} = 2\dot{a}, \quad \langle \tau_{g_h}^T \sigma_h \tau_{g_h}, 1 \rangle_{g_h,e_b} = 2\dot{b}, \quad \langle \tau_{g_h}^T \sigma_h \tau_{g_h}, 1 \rangle_{g_h,e_c} = 2\dot{c}.
\end{equation}
Indeed, the relations in~(\ref{dvdnabc}) follow from the fact that relative to $g_h$, $\nabla_{g_h} \phi_i$ is a vector of length $\frac{c}{2A}$ that is perpendicular to the edge $e_c$.  The relations in~(\ref{intabc}) follow from differentiating the identity
\begin{align*}
\langle \tau_{g_h}^T (g_h+t\sigma_h) \tau_{g_h}, 1 \rangle_{g_h,e_a}
&=  \left\langle \frac{1}{\tau^T g_h \tau} \tau^T(g_h + t\sigma_h) \tau, 1 \right\rangle_{g_h,e_a} \\
&=  \int_{e_a} \frac{\tau^T(g_h + t\sigma_h) \tau}{\sqrt{\tau^T g_h \tau}} \, d\ell \\
&= \frac{a(t)^2}{a(0)}
\end{align*}
and its counterpart for the edges $e_b$ and $e_c$.  Summing over all $K \in \omega_i$ and noting that $\phi_i=0$ on each $K \notin \omega_i$, we conclude that
\begin{equation} \label{thetadotjumps}
-\sum_{K \in \omega_i} \dot{\theta}_{iK} = \frac{1}{2} \sum_{e \in \mathcal{E}_h} \left \langle \tau_{g_h} \sigma_h \tau_{g_h}, \left\llbracket \frac{\partial \phi_i}{\partial n_{g_h}} \right\rrbracket \right\rangle_{g_h,e}.
\end{equation}
Since $\Hess_{g_h} \phi_i$ vanishes on each $K \in \mathcal{T}_h$, the relation~(\ref{dkappah}) follows.
\end{proof}

\begin{remark} \label{remark:christiansen}
Let us clarify the distinction between the preceding lemma and the results of Christiansen~\cite{christiansen2011linearization}. Christiansen works in three dimensions and computes the first- and second-order variation of the \emph{Regge action} $\sum_e \ell_e \kappa_e$ ($\kappa_e$ being the angle defect at an edge $e$, and $\ell_e$ being its length) around the Euclidean metric.  He notes that the first variation vanishes (a fact proven by Regge~\cite{regge1961general}), while the second variation is related to the distributional $\operatorname{curl} T \operatorname{curl}$ operator (which coincides with the operator $-\dv\dv S_\delta$ in two dimensions).  Along the way, he relates the linearization of $\kappa_e$ around the Euclidean metric to a summation of jumps of $\tau^T \sigma n$ (see~\cite[Proposition 2]{christiansen2011linearization}).  With some manipulation, this relation can be restated in the form of Theorem~\ref{thm:dkappah} with $g_h=\delta$.  Note that the case $g_h \neq \delta$ is not addressed in~\cite{christiansen2011linearization}; this missing ingredient plays a crucial role in our work.
\end{remark}

\section{Error Estimates} \label{sec:error}

In this section, we prove error estimates for $\kappa_h(g_h)-\kappa(g)$ in $H^k$-Sobolev norms of integer order $k \ge -1$.

Denote
\begin{align*}
\|v\|_{L^2(\Omega,g)}^2 &= \langle v, v \rangle_g = \int_\Omega v^2 \sqrt{\det g} \, dx, \\
|v|_{H^1(\Omega,g)}^2 &= \int_\Omega \nabla_g v^T g \nabla_g v \sqrt{\det g} \, dx = \int_\Omega \nabla v^T g^{-1} \nabla v \sqrt{\det g} \, dx,
\end{align*}
and $\|v\|_{H^1(\Omega,g)}^2 = \|v\|_{L^2(\Omega,g)}^2 + |v|_{H^1(\Omega,g)}^2$.  Also let
\begin{equation} \label{Hm1norm}
\|v\|_{H^{-1}(\Omega,g)} = \sup_{u \in H^1_0(\Omega)} \frac{\langle v, u \rangle_g}{\|u\|_{H^1(\Omega)}}.
\end{equation}

Note that if the eigenvalues of $g$ are bounded above and below by positive constants on $\overline{\Omega}$, then the norms $\|v\|_{L^2(\Omega,g)}$ and $\|v\|_{H^1(\Omega,g)}$ are equivalent to the norms $\|v\|_{L^2(\Omega)}$ and $\|v\|_{H^1(\Omega)}$, respectively.  These facts imply that~(\ref{Hm1norm}) is equivalent to the norm obtained by replacing $\|u\|_{H^1(\Omega)}$ with $\|u\|_{H^1(\Omega,g)}$ in the denominator of ~(\ref{Hm1norm}).

Our analysis will make use of the broken Sobolev norms
\[
\|v\|_{W^{s,p}_h(\Omega)} = \left( \sum_{K \in \mathcal{T}_h} \|v\|_{W^{s,p}(K)}^p \right)^{1/p},
\]
with the obvious modifications for $p=\infty$.  We denote $\|\cdot\|_{H^s_h(\Omega)} = \|\cdot\|_{W^{s,2}_h(\Omega)}$, and we use analogous notation $|\cdot|_{W^{s,p}_h(\Omega)}$ and $|\cdot|_{H^s_h(\Omega)}$ for the corresponding broken Sobolev semi-norms.

For $p \in [1,\infty]$ and $s > 2/p$, we denote
\[
\mathcal{M}^{s,p}(\Omega) = \{g \in W^{s,p}(\Omega) \otimes \mathbb{S} \mid g(x) \succ 0, \, \forall x \in \overline{\Omega} \}.
\]
Note that the condition $g(x) \succ 0$ is meaningful for $s>2/p$, since the Sobolev embedding theorem implies that elements of $W^{s,p}(\Omega)$ with $s > 2/p$ are continuous on $\overline{\Omega}$.   It is known that for $p \in (1,\infty)$ and $s > 2/p+1$, the curvature operator $\kappa$ maps $\mathcal{M}^{s,p}(\Omega)$ into $W^{s-2,p}(\Omega)$~\cite[Lemma 1]{fischer1975deformations}.  

\begin{theorem} \label{thm:err}
Let $g \in \mathcal{M}^{2,\infty}(\Omega)$.  Suppose that $\{g_h \in \Sigma_h\}_{h>0}$ is a sequence satisfying $\lim_{h \rightarrow 0} \|g_h-g\|_{L^\infty(\Omega)} = 0$, $\lim_{h \rightarrow 0} h^{-1} \log h^{-1} \|g_h-g\|_{L^2(\Omega)} = 0$, and $C_0 := \sup_{h > 0} \|g_h\|_{W^{1,\infty}_h(\Omega)} < \infty$.
Then there exists a constant $C$ depending on $\Omega$, $\|g\|_{W^{1,\infty}(\Omega)}$, $\|g^{-1}\|_{L^\infty(\Omega)}$, $\|\kappa(g)\|_{L^2(\Omega)}$, $r$, $q$, $C_0$, and the shape regularity and quasi-uniformity constants $C_1$ and $C_2$ such that for every $h$ sufficiently small,
\begin{equation} \label{Hm1est}
\begin{split}
h^{-1} &\|\kappa_h(g_h)-\kappa(g)\|_{H^{-1}(\Omega,g)} + \|  \kappa_h(g_h) - \kappa(g) \|_{L^2(\Omega)} \\ 
&\le C \left( h^{-2} \|g_h-g\|_{L^2(\Omega)} + h^{-1} |g_h-g|_{H^1_h(\Omega)} + \inf_{u_h \in V_h} \|\kappa(g) - u_h\|_{L^2(\Omega)} \right).
\end{split}
\end{equation}
Furthermore, if $\kappa(g) \in H^m(\Omega) \cap H^1_0(\Omega)$ with $m \in \mathbb{N}$, then for each $k=0,1,\dots,q$ and each $\ell=k,k+1,\dots,\min\{q+1,m\}$,
\begin{equation} \label{Hkest}
h^k | \kappa_h(g_h) - \kappa(g) |_{H^k_h(\Omega)} \le  C \bigg( h^{-2} \|g_h-g\|_{L^2(\Omega)} + h^{-1} |g_h-g|_{H^1_h(\Omega)} + h^\ell |\kappa(g)|_{H^\ell(\Omega)}  \bigg)
\end{equation}
for every $h$ sufficiently small.
\end{theorem}

Note that the theorem is vacuous when $r=0$, since we generally cannot expect $\|g_h-g\|_{L^2(\Omega)}$ to decay faster than $O(h^{r+1})$.  It should also be noted that the assumptions in the theorem statement guarantee that $g_h \in \Sigma_{h+}$ for $h$ sufficiently small, so that $\kappa_h(g_h)$ is well-defined.  See Section~\ref{sec:basic} for more details.

Our proof will be structured as follows.  First, we verify in Lemma~\ref{lemma:kappaint} that the exact Gaussian curvature satisfies
\[
\langle \kappa(g), v \rangle_g = \frac{1}{2} \int_0^1 b_h( (1-t)\delta + t g; g - \delta, v) \, dt
\]
for every $v \in V$.  Next, we consider an arbitrary $v \in H^1_0(\Omega)$ and write
\begin{align}
\langle \kappa_h(g_h) - \kappa(g), v \rangle_g 
&= \big( \langle \kappa_h(g_h), v_h \rangle_{g_h} - \langle \kappa(g), v_h \rangle_g \big) \nonumber \\
&\;\;\; + \langle \kappa_h(g_h) - \kappa(g), v - v_h \rangle_g \label{spliterror} \\
&\;\;\; + \big( \langle \kappa_h(g_h), v_h \rangle_g - \langle \kappa_h(g_h), v_h \rangle_{g_h} \big), \nonumber
\end{align}
with $v_h \in V_h$.
The three terms above are estimated in Propositions~\ref{prop1},~\ref{prop2}, and~\ref{prop3}.  Combining them will lead to Theorem~\ref{thm:err}.

\subsection{Consistency of $b_h$}

We begin by recalling two integration-by-parts formulas.
\begin{lemma}
Let $K \in \mathcal{T}_h$.
For every $v \in H^1(K)$ and every $w \in H^1(K) \otimes \mathbb{R}^2$, we have
\begin{equation} \label{ibp1}
\int_K  w^T g \nabla_g v \sqrt{\det g} \, dx = \int_{\partial K} w^T g n_g v \sqrt{\tau^T g \tau} \, d\ell - \int_K (\dv_g w) v \sqrt{\det g} \, dx.
\end{equation}
In addition, for every $v \in H^2(K)$ and every $\sigma \in H^1(K) \otimes \mathbb{S}$, we have
\begin{equation} \label{ibp2}
\begin{split}
\int_K (\dv_g \sigma)^T g \nabla_g v \sqrt{\det g} \, dx = \int_{\partial K} n_g^T \sigma \nabla_g v \sqrt{\tau^T g \tau} \, d\ell \\ - \int_K \Tr(g^{-1} \sigma g^{-1} \Hess_g v) \sqrt{\det g} \, dx.
\end{split}
\end{equation}
\end{lemma}
\begin{proof}
The (Euclidean) divergence theorem gives
\begin{align*}
\int_K w^T g \nabla_g v \sqrt{\det g} \, dx 
&= \int_K w^T \nabla v \sqrt{\det g} \, dx \\
&= \int_{\partial K} w^T n v \sqrt{\det g} \, d\ell - \int_K \dv(w \sqrt{\det g}) v \, dx \\
&= \int_{\partial K} w^T g n_g v \sqrt{\tau^T g \tau} \, d\ell - \int_K (\dv_g w) v \sqrt{\det g} \, dx,
\end{align*}
where we have used the fact that
\begin{equation} \label{gng}
g n_g = \frac{ gJg \tau}{\sqrt{\tau^T g \tau} \sqrt{\det g}} = \frac{\sqrt{\det g} \, J \tau}{\sqrt{\tau^T g \tau}} = \frac{\sqrt{\det g} \, n}{\sqrt{\tau^T g \tau}}.
\end{equation}
To prove~(\ref{ibp2}), note first that
\begin{equation} \label{divrelation}
(\dv_g \sigma)^T g \nabla_g v = \dv_g(g^{-1} \sigma \nabla_g v) - \Tr(g^{-1} \sigma g^{-1} \Hess_g v),
\end{equation}
which can be verified by expanding both sides in coordinates.  
The formula~(\ref{ibp2}) now follows from the observation that
\begin{align*}
\int_K \dv_g(g^{-1}\sigma \nabla_g v) \sqrt{\det g} \, dx
&= \int_K \dv(g^{-1} \sigma \nabla_g v \sqrt{\det g}) \, dx \\
&= \int_{\partial K} n^T g^{-1} \sigma \nabla_g v \sqrt{\det g} \, d\ell \\
&= \int_{\partial K} n_g^T \sigma \nabla_g v \sqrt{\tau^T g \tau} \, d\ell,
\end{align*}
where we have used~(\ref{gng}) again.  
\end{proof}

\begin{lemma} \label{intbyparts}
For every $\sigma \in H^2(\Omega) \otimes \mathbb{S}$ and every $v \in V$,
\[
b_h(g; \sigma, v) = \langle \dv_g \dv_g S_g \sigma, v \rangle_g.
\]
\end{lemma}
\begin{proof}
Let $\widetilde{\sigma} = S_g \sigma$.  Recalling that $\tau_g^T \sigma \tau_g = -n_g^T \widetilde{\sigma} n_g$, we have
\begin{align*}
b_h(g; \sigma, v) 
&= \sum_{K \in \mathcal{T}_h} \bigg( \int_K \Tr(g^{-1}\widetilde{\sigma} g^{-1} \Hess_g v) \sqrt{\det g} \, dx - \int_{\partial K} n_g^T \widetilde{\sigma} n_g n_g^T g \nabla_g v \sqrt{\tau^T g \tau} \, d\ell \bigg) \\
&= \sum_{K \in \mathcal{T}_h} \bigg( \int_{\partial K} n_g^T \widetilde{\sigma} \nabla_g v \sqrt{\tau^T g \tau} \, d\ell - \int_K (\dv_g \widetilde{\sigma})^T g \nabla_g v \sqrt{\det g} \, dx \\
&\hspace{0.5in} - \int_{\partial K} n_g^T \widetilde{\sigma} n_g n_g^T g \nabla_g v \sqrt{\tau^T g \tau} \, d\ell  \bigg).
\end{align*}
Using the fact that $\tau_g \tau_g^T + n_g n_g^T = g^{-1}$, the first and third terms can be combined to give
\begin{align*}
b_h(g; \sigma, v) &= \sum_{K \in \mathcal{T}_h} \bigg( \int_{\partial K} n_g^T \widetilde{\sigma} \tau_g \tau_g^T g \nabla_g v \sqrt{\tau^T g \tau} \, d\ell - \int_K (\dv_g \widetilde{\sigma})^T g \nabla_g v \sqrt{\det g} \, dx \bigg).
\end{align*}
Integrating the second term by parts, we obtain
\begin{align}
b_h(g; \sigma, v) 
&= \sum_{K \in \mathcal{T}_h} \bigg( \int_{\partial K} n_g^T \widetilde{\sigma} \tau_g \tau_g^T g \nabla_g v \sqrt{\tau^T g \tau} \, d\ell - \int_{\partial K} (\dv_g \widetilde{\sigma})^T g n_g v \sqrt{\tau^T g \tau} \, d\ell \nonumber \\
&\hspace{0.5in} + \int_K (\dv_g \dv_g \widetilde{\sigma}) v \sqrt{\det g} \, dx \bigg) \nonumber \\
&= \sum_{e \in \mathcal{E}_h} \int_e \llbracket n_g^T \widetilde{\sigma} \tau_g  \tau_g^T g \nabla_g v  \rrbracket \sqrt{\tau^T g \tau} \, d\ell - \sum_{e \in \mathcal{E}_h} \int_e \llbracket (\dv_g \widetilde{\sigma})^T g n_g  v \rrbracket \sqrt{\tau^T g \tau} \, d\ell \nonumber \\
&\hspace{0.5in} + \int_\Omega (\dv_g \dv_g \widetilde{\sigma}) v \sqrt{\det g} \, dx. \label{sum}
\end{align}
Since $v \in V$, both $v$ and its tangential derivative are continuous across element interfaces, and they vanish on $\partial\Omega$.  In particular, $\llbracket \tau_g^T g \nabla_g v \rrbracket = \llbracket \tau^T \nabla v / \sqrt{\tau^T g \tau} \rrbracket = 0$ for every $e \in \mathcal{E}_h$.   On the other hand, since $\widetilde{\sigma} \in H^2(\Omega) \otimes \mathbb{S}$, $n_g^T \widetilde{\sigma} \tau_g$ is continuous across element interfaces, and $\llbracket (\dv_g \widetilde{\sigma})^T g n_g \rrbracket = 0$ for every $e \in \mathring{\mathcal{E}}_h$.  It follows that the summations over $e \in \mathcal{E}_h$ in~(\ref{sum}) vanish.  Thus,
\[
b_h(g; \sigma, v) =  \int_\Omega (\dv_g \dv_g \widetilde{\sigma}) v \sqrt{\det g} \, dx.
\]
\end{proof}

\begin{lemma} \label{lemma:kappaint}
For every $v \in V$,
\[
\langle \kappa(g), v \rangle_g = \frac{1}{2} \int_0^1 b_h( (1-t)\delta + t g; g - \delta, v) \, dt.
\]
\end{lemma}
\begin{proof}
Let $G(t) = (1-t)\delta + tg$ and $\sigma = g-\delta$.  From~(\ref{linint}), we obtain
\begin{align*}
\frac{d}{dt} \langle \kappa(G(t)), v \rangle_{G(t)} 
&= \frac{1}{2} \langle \dv_{G(t)} \dv_{G(t)} S_{G(t)} \sigma, v \rangle_{G(t)} \\
&= \frac{1}{2} b_h(G(t); \sigma, v),
\end{align*}
so 
\begin{align*}
\langle \kappa(g), v \rangle_{g}  
&= \langle \kappa(\delta), v \rangle_{\delta} + \frac{1}{2} \int_0^1 b_h(G(t); \sigma, v) \, dt \\
&= \frac{1}{2} \int_0^1 b_h(G(t); \sigma, v) \, dt.
\end{align*}
\end{proof}

\subsection{Basic Estimates} \label{sec:basic}

We now collect a few basic estimates in preparation for our estimation of $\langle \kappa_h(g_h) - \kappa(g), v \rangle_g$.

Throughout what follows, we make use of the fact that the eigenvalues of $g$, being positive and continuous on $\overline{\Omega}$, are bounded above and below by positive constants $C_3$ and $C_4$ that depend on $\|g\|_{L^\infty(\Omega)}$ and $\|g^{-1}\|_{L^\infty(\Omega)}$.  Hence, for any $w \in \mathbb{R}^2$, we have
\begin{equation} \label{rayleigh}
C_3 w^T w \le w^T g(x) w \le C_4 w^T w, \quad \forall x \in \overline{\Omega}.
\end{equation}
It follows also that $\|\sqrt{\det g}\|_{L^\infty(\Omega)}$ and $\|1/\sqrt{\det g}\|_{L^\infty(\Omega)}$ are each bounded above by constants depending on $\|g\|_{L^\infty(\Omega)}$ and $\|g^{-1}\|_{L^\infty(\Omega)}$.  Differentiation then reveals that $\|\sqrt{\det g}\|_{W^{1,\infty}(\Omega)}$ and $\|1/\sqrt{\det g}\|_{W^{1,\infty}(\Omega)}$ are each bounded above by constants depending on $\|g\|_{W^{1,\infty}(\Omega)}$ and $\|g^{-1}\|_{L^\infty(\Omega)}$.

We will use this information to establish analogous estimates for $g_h$, under the assumption that $C_0 := \sup_{h>0} \|g_h\|_{W^{1,\infty}_h} < \infty$ and $\lim_{h \rightarrow 0} \|g_h-g\|_{L^\infty(\Omega)} = 0$.  From this point forward, we use the letter $C$ to denote a constant which is not necessarily the same at each occurrence and may depend on $\Omega$, $\|g\|_{W^{1,\infty}(\Omega)}$, $\|g^{-1}\|_{L^\infty(\Omega)}$, $\|\kappa(g)\|_{L^2(\Omega)}$, $r$, $q$, $C_0$, and the shape regularity and quasi-uniformity constants $C_1$ and $C_2$, but not on $h$.

\begin{lemma} \label{lemma:ghinv}
Let $p \in [1,\infty]$.  For every $h$ sufficiently small,
\[
\|g_h^{-1} - g^{-1}\|_{L^p(\Omega)} \le C \|g_h - g\|_{L^p(\Omega)} 
\]
\end{lemma}
\begin{proof}
The identity
\[
(g_h^{-1} - g^{-1}) (\delta - (g-g_h)g^{-1}) = g^{-1} (g-g_h) g^{-1} 
\]
shows that for $h$ sufficiently small,
\[
g_h^{-1} - g^{-1} = g^{-1} (g-g_h) g^{-1} (\delta - (g-g_h)g^{-1})^{-1}.
\]
Thus, for every $K \in \mathcal{T}_h$ and every $x \in \mathring{K}$, the operator norm of $g_h(x)^{-1} - g(x)^{-1}$ satisfies
\[
\|g_h(x)^{-1} - g(x)^{-1}\| \le \frac{ \|g^{-1}(x) (g(x)-g_h(x)) g(x)^{-1}\| }{ 1 - \|(g(x)-g_h(x))g(x)^{-1}\| }.
\]
It follows that
\[
\|g_h^{-1} - g^{-1}\|_{L^p(\Omega)} \le \frac{ \|g^{-1}\|_{L^\infty(\Omega)}^2 \|g-g_h\|_{L^p(\Omega)} }{ 1- \|g-g_h\|_{L^\infty(\Omega)} \|g^{-1}\|_{L^\infty(\Omega)} }
\]
for $h$ sufficiently small.  Taking $h$ small enough so that (say) $\|g-g_h\|_{L^\infty(\Omega)} \le \frac{1}{2\|g^{-1}\|_{L^\infty(\Omega)}}$ completes the proof.
\end{proof}

From the above lemma and the assumption that $\sup_{h>0} \|g_h\|_{W^{1,\infty}_h} < \infty$, we conclude that for $h$ sufficiently small,
\[
\|g_h\|_{W^{1,\infty}_h(\Omega)} + \|g_h^{-1}\|_{L^\infty(\Omega)} \le C.
 \]
In particular, the eigenvalues of $g_h$ are bounded above and below by positive constants independent of $h$, and estimates of the form
\[
C^{-1} w^T w \le w^T g_h(x) w \le C w^T w, \quad \forall x \in \mathring{K}, \, \forall K \in \mathcal{T}_h, \, \forall w \in \mathbb{R}^2,
\]
and
\[
\|\sqrt{\det g_h}\|_{L^\infty(\Omega)} + \|1/\sqrt{\det g_h}\|_{L^{\infty}(\Omega)} \le C
\]
hold for $h$ sufficiently small.  In the remainder of this section, we will tacitly assume that $h$ is small enough that the preceding estimates hold.

\begin{lemma} \label{lemma:sqrtdetgWkp}
For every $k \in \{0,1\}$ and $p \in [1,\infty]$,
\[
\| \sqrt{\det g} - \sqrt{\det g_h} \|_{W^{k,p}_h(\Omega)} \le C \|g_h-g\|_{W^{k,p}_h(\Omega)}.
\]
\end{lemma}
\begin{proof}
Let $g = \left( \begin{smallmatrix} a & b \\ c & d \end{smallmatrix} \right)$ and $g_h = \left( \begin{smallmatrix} a_h & b_h \\ c_h & d_h \end{smallmatrix} \right)$.  The identity
\[
\det g_h - \det g = (a_h - a)d_h + a(d_h-d) - (b_h-b)c_h - b(c_h-c)
\]
shows that
\[
\|\det g - \det g_h\|_{L^p(\Omega)} \le C\|g_h-g\|_{L^p(\Omega)},
\]
and the identity
\[
\sqrt{\det g} - \sqrt{\det g_h} = \frac{\det g - \det g_h}{\sqrt{\det g} + \sqrt{\det g_h}}
\]
shows that
\begin{equation} \label{sqrtdetgLp}
\|\sqrt{\det g} - \sqrt{\det g_h}\|_{L^p(\Omega)} \le C\|\det g - \det g_h\|_{L^p(\Omega)} \le C\|g_h-g\|_{L^p(\Omega)}.
\end{equation}
To obtain the $W^{1,p}_h(\Omega)$ estimate, we compute
\[
\frac{\partial}{\partial x^i} \sqrt{\det g} = \frac{1}{2} \Tr\left( g^{-1} \frac{\partial g}{\partial x^i} \right) \sqrt{\det g},
\]
which gives
\begin{dmath*}
\frac{\partial}{\partial x^i} (\sqrt{\det g} - \sqrt{\det g_h}) = \frac{1}{2} \Tr\left( \left(g^{-1}-g_h^{-1}\right) \frac{\partial g}{\partial x^i} \right) \sqrt{\det g} + \frac{1}{2} \Tr\left( g_h^{-1} \left(  \frac{\partial g}{\partial x^i} -  \frac{\partial g_h}{\partial x^i} \right) \right) \sqrt{\det g} +  \frac{1}{2} \Tr\left( g_h^{-1} \frac{\partial g_h}{\partial x^i} \right) (\sqrt{\det g} - \sqrt{\det g_h}).
\end{dmath*}
From this,~(\ref{sqrtdetgLp}), and Lemma~\ref{lemma:ghinv}, we conclude that
\begin{align*}
| \sqrt{\det g} - \sqrt{\det g_h} |_{W^{1,p}_h(\Omega)} 
&\le C \bigg( \|g^{-1}-g_h^{-1}\|_{L^p(\Omega)} + |g-g_h|_{W^{1,p}_h(\Omega)} \\&\hspace{0.5in} + \|\sqrt{\det g} - \sqrt{\det g_h}\|_{L^p(\Omega)} \bigg) \\
&\le C\|g-g_h\|_{W^{1,p}_h(\Omega)}.
\end{align*}
\end{proof}

\subsubsection{The $L^2(\Omega,g)$-Orthogonal Projector}

Let $P_h : L^2(\Omega) \rightarrow V_h$ denote the $L^2(\Omega,g)$-orthogonal projector onto $V_h$, so that
\[
\langle P_h u - u, w_h \rangle_g = 0, \quad \forall u \in L^2(\Omega), \, \forall w_h \in V_h.
\]
\begin{lemma} \label{lemma:Pherr}
We have
\begin{align}
\|P_h u - u\|_{H^{-1}(\Omega,g)} &\le C h \inf_{u_h \in V_h} \|u_h - u\|_{L^2(\Omega)}, &\quad &\forall u \in L^2(\Omega), \label{PherrHm1} \\
\|P_h u - u\|_{L^2(\Omega)} &\le C \inf_{u_h \in V_h} \|u_h - u\|_{L^2(\Omega)}, &\quad &\forall u \in L^2(\Omega), \label{PherrL2} \\
\|P_h u\|_{H^1(\Omega)} &\le C\|u\|_{H^1(\Omega)}, &\quad &\forall u \in H^1_0(\Omega). \label{PhboundedH1}
\end{align}
\end{lemma}
\begin{proof}
For any $u \in L^2(\Omega)$ and any $u_h \in V_h$, we have
\[
\|P_h u - u \|_{L^2(\Omega)} \le \|P_h(u-u_h)\|_{L^2(\Omega)} + \|u_h-u\|_{L^2(\Omega)}.
\]
The equivalence of the norms $\|\cdot\|_{L^2(\Omega,g)}$ and $\|\cdot\|_{L^2(\Omega)}$ implies
\[
\|P_h(u-u_h)\|_{L^2(\Omega)} \le C \|P_h(u-u_h)\|_{L^2(\Omega,g)} \le C \|u-u_h\|_{L^2(\Omega,g)} \le C \|u-u_h\|_{L^2(\Omega)},
\]
so $\|P_h u - u \|_{L^2(\Omega)} \le C \|u_h-u\|_{L^2(\Omega)}$.  Since $u_h$ was arbitrary,~(\ref{PherrL2}) follows.  It follows also that $\|P_h u - u\|_{L^2(\Omega)} \le C h |u|_{H^1(\Omega)}$ if $u \in H^1_0(\Omega)$.  Now let $Q_h : H^1(\Omega) \rightarrow V_h$ denote the Scott-Zhang interpolation operator~\cite{scott1990finite}.  
Using an inverse estimate, interpolation estimates, and the stability of $Q_h$ in $H^1(\Omega)$, we obtain
\begin{align*}
\|P_h u\|_{H^1(\Omega)} 
&\le \|P_h u - Q_h u\|_{H^1(\Omega)} + \|Q_h u\|_{H^1(\Omega)} \\
&\le C\left( h^{-1} \|P_h u - Q_h u\|_{L^2(\Omega)} + \|u\|_{H^1(\Omega)} \right) \\
&\le C\left( h^{-1} \|P_h u - u\|_{L^2(\Omega)} + h^{-1} \|u - Q_h u\|_{L^2(\Omega)} + \|u\|_{H^1(\Omega)}  \right) \\
&\le C\|u\|_{H^1(\Omega)}, \quad \forall u \in H^1_0(\Omega).
\end{align*}
Finally, for any $v \in H^1_0(\Omega)$,
\begin{align*}
|\langle P_h u - u, v \rangle_g| 
&= |\langle P_h u - u, v-P_h v \rangle_g| \\
&\le \|P_h u - u\|_{L^2(\Omega,g)} \| v-P_h v \|_{L^2(\Omega,g)} \\
&\le C \|P_h u - u\|_{L^2(\Omega)} \| v-P_h v \|_{L^2(\Omega)} \\
&\le C \|P_h u - u\|_{L^2(\Omega)} h | v |_{H^1(\Omega)},
\end{align*}
so
\[
\|P_h u - u\|_{H^{-1}(\Omega,g)} \le C h \|P_h u - u\|_{L^2(\Omega)} \le C h \inf_{u_h \in V_h} \|u_h - u\|_{L^2(\Omega)}.
\]
\end{proof}

\subsection{Proof of Theorem~\ref{thm:err}}

We are now ready to estimate the three terms in~(\ref{spliterror}).
\begin{proposition} \label{prop1}
If $v_h = P_h v$, then
\[
|\langle \kappa_h(g_h) - \kappa(g), v - v_h \rangle_g| \le C h \|v\|_{H^1(\Omega)} \inf_{u_h \in V_h} \| \kappa(g) - u_h \|_{L^2(\Omega)}.
\]
\end{proposition}
\begin{proof}
For any $u_h \in V_h$, we have
\[
\langle \kappa_h(g_h) - \kappa(g), v - P_h v \rangle_g = \langle u_h -  \kappa(g), v - P_h v \rangle_g,
\]
so
\begin{align*}
|\langle \kappa_h(g_h) - \kappa(g), v - P_h v \rangle_g| 
&\le \| u_h - \kappa(g) \|_{H^{-1}(\Omega,g)} \|v - P_h v\|_{H^1(\Omega)} \\
&\le \| u_h - \kappa(g) \|_{H^{-1}(\Omega,g)} \left(\|v\|_{H^1(\Omega)} + \|P_h v\|_{H^1(\Omega)}\right) \\
&\le C \| u_h - \kappa(g) \|_{H^{-1}(\Omega,g)} \|v\|_{H^1(\Omega)}.
\end{align*}
Taking $u_h = P_h \kappa(g)$ and invoking the $H^{-1}(\Omega,g)$-estimate~(\ref{PherrHm1}) completes the proof.
\end{proof}

\begin{proposition} \label{prop2}
For $h$ sufficiently small, we have
\begin{dmath*}
|\langle \kappa_h(g_h), v_h \rangle_g - \langle \kappa_h(g_h), v_h \rangle_{g_h}|  \\ \le C  \log h^{-1} \left(1 + h^{-1} \|\kappa_h(g_h)-\kappa(g)\|_{H^{-1}(\Omega,g)}\right) \|g_h-g\|_{L^2(\Omega)} \|v_h\|_{H^1(\Omega)}.
\end{dmath*}
\end{proposition}
\begin{proof}
By Lemma~\ref{lemma:sqrtdetgWkp} and the inverse estimate $\|v_h\|_{L^\infty(\Omega)} \le C \log h^{-1} \|v_h\|_{H^1(\Omega)}$~\cite[Lemma 4.9.2]{brenner2007mathematical}, we have
\begin{align}
\big| \langle \kappa_h(g_h), v_h \rangle_g &- \langle \kappa_h(g_h), v_h \rangle_{g_h} \big| \nonumber \\
&= \left| \int_\Omega \kappa_h(g_h) v_h (\sqrt{\det g} - \sqrt{\det g_h}) \, dx \right| \nonumber \\
&\le \| \kappa_h(g_h) \|_{L^2(\Omega)} \|v_h\|_{L^\infty(\Omega)} \| \sqrt{\det g} - \sqrt{\det g_h}  \|_{L^2(\Omega)} \nonumber \\
&\le  C \log h^{-1} \| \kappa_h(g_h) \|_{L^2(\Omega)} \|v_h\|_{H^1(\Omega)} \| g_h - g \|_{L^2(\Omega)}. \label{bound1}
\end{align}  
Now note that
\begin{align*}
\| \kappa_h(g_h) \|_{L^2(\Omega)} 
&\le \| \kappa_h(g_h) - P_h \kappa(g) \|_{L^2(\Omega)} + \|P_h \kappa(g)\|_{L^2(\Omega)}. 
\end{align*}
The second term above is bounded by a constant (a multiple of $\|\kappa(g)\|_{L^2(\Omega)}$) owing to~(\ref{PherrL2}).  For the first term, an inverse estimate gives
\begin{align*}
\|\kappa_h(g_h)-P_h \kappa(g)\|_{L^2(\Omega)}^2 
&\le C \|\kappa_h(g_h)-P_h \kappa(g)\|_{L^2(\Omega,g)}^2  \\ 
&= C\langle \kappa_h(g_h) -P_h \kappa(g), \kappa_h(g_h)-P_h \kappa(g) \rangle_g \\
&\le C \|\kappa_h(g_h)-P_h \kappa(g)\|_{H^{-1}(\Omega,g)} \|\kappa_h(g_h)-P_h \kappa(g)\|_{H^1(\Omega)} \\
&\le C h^{-1} \|\kappa_h(g_h)-P_h \kappa(g)\|_{H^{-1}(\Omega,g)} \|\kappa_h(g_h)-P_h \kappa(g)\|_{L^2(\Omega)}.
\end{align*}
Thus, using the $H^{-1}(\Omega,g)$-estimate~(\ref{PherrHm1}), we obtain
\begin{align*}
\|\kappa_h(g_h)- & P_h \kappa(g)\|_{L^2(\Omega)} 
\le C h^{-1} \|\kappa_h(g_h)-P_h \kappa(g)\|_{H^{-1}(\Omega,g)} \\
&\le C \left( h^{-1}\| \kappa_h(g_h) - \kappa(g)\|_{H^{-1}(\Omega,g)} + h^{-1} \| \kappa(g) - P_h \kappa(g)\|_{H^{-1}(\Omega,g)} \right) \\
&\le C \left( h^{-1} \| \kappa_h(g_h) - \kappa(g)\|_{H^{-1}(\Omega,g)} + \inf_{u_h \in V_h} \| \kappa(g) - u_h\|_{L^2(\Omega)} \right) \\
&\le C \left( h^{-1} \| \kappa_h(g_h) - \kappa(g)\|_{H^{-1}(\Omega,g)} + 1 \right),
\end{align*}
for $h$ sufficiently small, since $\lim_{h \rightarrow 0} \inf_{u_h \in V_h} \| \kappa(g) - u_h\|_{L^2(\Omega)} = 0$~\cite[Equation 1.99]{ern2004theory}.  It follows that
\[
\|\kappa_h(g_h)\|_{L^2(\Omega)} \le C \left( h^{-1} \| \kappa_h(g_h) - \kappa(g)\|_{H^{-1}(\Omega,g)} + 1 \right).
\]
Combining this with~(\ref{bound1}) completes the proof.
\end{proof}

We will now estimate the first term in~(\ref{spliterror}).  We state the result below, and prove it with a series of lemmas.
\begin{proposition} \label{prop3}
We have
\begin{equation*}\begin{split}
| \langle \kappa_h(g_h), v_h \rangle_{g_h} & - \langle \kappa(g), v_h \rangle_g | \\
&\le C \left( h^{-1} \|g_h-g\|_{L^2(\Omega)} + |g_h-g|_{H^1_h(\Omega)} \right) \left( |v_h|_{H^1(\Omega)} + h |v_h|_{H^2_h(\Omega)} \right).
\end{split}\end{equation*}
\end{proposition}

To prove Proposition~\ref{prop3}, let $G(t) = (1-t)\delta + tg$ and $G_h(t)=(1-t)\delta+tg_h$, so that
\begin{equation*}\begin{split}
\langle \kappa_h(g_h), v_h \rangle_{g_h} - \langle \kappa(g), v_h \rangle_g
= \frac{1}{2}\int_0^1 b_h(G_h(t); g_h - \delta, v_h) - b_h(G(t); g_h-\delta, v_h) \, dt \\ + \frac{1}{2} \int_0^1 b_h(G(t); g_h-g, v_h) \, dt.
\end{split}\end{equation*}
Note that $G_h(t)$, being a convex combination of $\delta$ and $g_h$, has eigenvalues bounded above and below by positive constants independent of $h$ (and $t$), and all of the estimates we established for the norms of $g_h-g$, $g_h^{-1}-g^{-1}$, $\sqrt{\det g_h} - \sqrt{\det g}$, etc. carry over easily to $G_h-G$, $G_h^{-1}-G^{-1}$, $\sqrt{\det G_h} - \sqrt{\det G}$, etc.

\begin{lemma} \label{lemma:bhdiff}
For every $t \in [0,1]$,
\begin{equation*}\begin{split}
|b_h(G_h(t); g_h - \delta, v_h) &- b_h(G(t); g_h - \delta, v_h)| \\
&\le C \left( h^{-1} \|g_h-g\|_{L^2(\Omega)} + |g_h-g|_{H^1_h(\Omega)} \right) \left( |v_h|_{H^1(\Omega)} + h |v_h|_{H^2_h(\Omega)} \right).
\end{split}\end{equation*}
\end{lemma}
\begin{proof}
Let $\sigma = g_h -\delta$.  We have
\[
b_h(G_h(t); \sigma, v_h) - b_h(G(t); \sigma, v_h) = \frac{1}{2} \left( \sum_{e \in \mathcal{E}_h} A_e + \sum_{K \in \mathcal{T}_h} B_K \right),
\]
where
\begin{align*}
A_e &= \left\langle \tau_{G_h}^T \sigma \tau_{G_h}, \bigg\llbracket \frac{\partial v_h}{\partial n_{G_h}} \bigg\rrbracket \right\rangle_{G_h,e}  - \left\langle \tau_G^T \sigma \tau_G, \bigg\llbracket \frac{\partial v_h}{\partial n_G} \bigg\rrbracket \right\rangle_{G,e},  \\
B_K &= \langle S_{G_h} \sigma, \Hess_{G_h} v_h \rangle_{G_h,K} - \langle S_G \sigma, \Hess_G v_h \rangle_{G,K}.
\end{align*}

Along any edge $e \in \mathcal{E}_h$, we may write
\[
\left\langle \tau_G^T \sigma \tau_G, \left\llbracket \frac{\partial v_h}{\partial n_G} \right\rrbracket \right\rangle_{G,e} = \int_e \tau^T \sigma \tau \llbracket \nabla v_h^T J G (\det G)^{-1/2} (\tau^T G \tau)^{-1} \tau \rrbracket  \, d\ell,
\]
so, with $\widetilde{G} = G (\det G)^{-1/2} (\tau^T G \tau)^{-1}$ and $\widetilde{G}_h = G_h (\det G_h)^{-1/2} (\tau^T G_h \tau)^{-1}$, we have 
\[
A_e = \int_e \tau^T \sigma \tau \llbracket \nabla v_h^T J (\widetilde{G}_h - \widetilde{G}) \tau \rrbracket \, d\ell.
\]
Writing
\begin{align*}
\widetilde{G} - \widetilde{G}_h &= G (\det G)^{-1/2} \left( \frac{ (\tau^T G_h \tau) - (\tau^T G \tau) }{ (\tau^T G_h \tau) (\tau^T G \tau) } \right) 
\\&\;\;\; +  G \left( \frac{\sqrt{\det G_h} - \sqrt{\det G}}{\sqrt{\det G_h}\sqrt{\det G}} \right) (\tau^T G_h \tau)^{-1} 
\\&\;\;\; + (G-G_h)(\det G_h)^{-1/2} (\tau^T G_h \tau)^{-1} 
\end{align*}
shows that for each $K_j \in \mathcal{T}_h$ with $e \subset \partial K_j$,
\[
\| (\widetilde{G} - \widetilde{G}_h)\big|_{\partial K_j} \|_{L^2(e)} \le C \| \left.(G_h-G)\right|_{\partial K_j} \|_{L^2(e)}.
\]
The trace inequality then gives, if $e \in \mathring{\mathcal{E}}_h$,
\begin{align}
|A_e|
&\le C\|\tau^T \sigma \tau\|_{L^\infty(e)} \left( \left\| \left.\nabla v_h\right|_{\partial K_1} \right\|_{L^2(e)} +  \left\| \left.\nabla v_h\right|_{\partial K_2} \right\|_{L^2(e)} \right) \nonumber\\
&\;\;\; \times \left( \| (G-G_h)\big|_{\partial K_1} \|_{L^2(e)} + \| (G-G_h)\big|_{\partial K_2} \|_{L^2(e)} \right) \nonumber\\
&\le C\|\tau^T \sigma \tau\|_{L^\infty(e)} \bigg( h^{-1/2} |v_h|_{H^1(K_1)} + h^{1/2} |v_h|_{H^2(K_1)} \nonumber\\&\hspace{0.5in} + h^{-1/2} |v_h|_{H^1(K_2)} + h^{1/2} |v_h|_{H^2(K_2)} \bigg) \nonumber\\
&\;\;\; \times \bigg( h^{-1/2} \|G_h - G\|_{L^2(K_1)} + h^{1/2} |G_h-G|_{H^1(K_1)} \nonumber\\&\hspace{0.5in} + h^{-1/2} \|G_h-G\|_{L^2(K_2)} + h^{1/2} |G_h-G|_{H^1(K_2)} \bigg). \label{edgeintbound}
\end{align}
For edges $e$ on $\partial \Omega$, the same estimate holds with terms involving $K_2$ replaced by zero.

Turning now to $B_K$, observe that
\begin{align}
\langle S_G \sigma, \Hess_G v_h \rangle_{G,K}
&= \langle \sigma - G \Tr(G^{-1}\sigma), \Hess_G v_h \rangle_{G,K} \nonumber \\
&= \int_K \Tr( G^{-1} \sigma G^{-1} \Hess_G v_h) \sqrt{\det G} \, dx 
\\&\;\;\; - \int_K \Tr( G^{-1} \sigma) \Tr( G^{-1} \Hess_G v_h ) \sqrt{\det G} \, dx. \label{SGsplit}
\end{align}
Observe also that
\[
\Tr( G^{-1} \sigma G^{-1} \Hess_G v_h) = \Tr( \sigma G^{-1} (\Hess_G v_h) G^{-1} ) = \sigma_{\ell k}{ a^{\ell k} },
\]
where
\begin{align*}
a^{\ell k} = G^{\ell i} \left( \frac{\partial^2 v_h}{\partial x^i \partial x^j} - \Gamma^m_{ij} \frac{\partial v_h}{\partial x^m} \right) G^{jk},
\end{align*}
and $\Gamma_{ij}^m$ are the Christoffel symbols associated with $G$.
If $a_h^{\ell k}$ is defined similarly with $G_h$ in place of $G$, then it is not hard to see that
\[
\|a_h^{\ell k} - a^{\ell k}\|_{L^1(K)} \le C \left( \|G_h-G\|_{L^2(K)}  |v_h|_{H^2(K)} + \|G_h-G\|_{H^1(K)}  |v_h|_{H^1(K)} \right)
\]
and
\[
\|a^{\ell k}\|_{L^2(K)} \le C \left( |v_h|_{H^2(K)} + |v_h|_{H^1(K)} \right),
\]
so
\begin{align*}
\bigg| \int_K & \Tr( G_h^{-1} \sigma G_h^{-1} \Hess_{G_h} v_h) \sqrt{\det G_h} \, dx  - \int_K \Tr( G^{-1} \sigma G^{-1} \Hess_G v_h) \sqrt{\det G} \, dx  \bigg| \\
&\le \left| \int_K \sigma_{\ell k} (a_h^{\ell k} - a^{\ell k})  \sqrt{\det G_h} \, dx\right|  + \left| \int_K \sigma_{\ell k} a^{\ell k} \left(\sqrt{\det G_h} - \sqrt{\det G}\right) \, dx \right| \\
&\le C \|\sigma\|_{L^\infty(K)} \left( \|G_h-G\|_{L^2(K)}  |v_h|_{H^2(K)} + \|G_h-G\|_{H^1(K)}  |v_h|_{H^1(K)} \right).
\end{align*}
Using a similar argument for the second term in~(\ref{SGsplit}), we get
\begin{equation*}\begin{split}
\left| \int_K \Tr( G_h^{-1} \sigma) \Tr( G_h^{-1} \Hess_{G_h} v_h ) \sqrt{\det G_h} \, dx - \int_K \Tr( G^{-1} \sigma) \Tr( G^{-1} \Hess_G v_h ) \sqrt{\det G} \, dx \right| \\
\le C \|\sigma\|_{L^\infty(K)} \left( \|G_h-G\|_{L^2(K)}  |v_h|_{H^2(K)} + \|G_h-G\|_{H^1(K)}  |v_h|_{H^1(K)} \right).
\end{split}\end{equation*}
Hence,
\begin{equation} \label{BKbound}
|B_K| \le C \|\sigma\|_{L^\infty(K)} \left( \|G_h-G\|_{L^2(K)}  |v_h|_{H^2(K)} + \|G_h-G\|_{H^1(K)}  |v_h|_{H^1(K)} \right).
\end{equation}
The proof is completed by summing~(\ref{edgeintbound}) over all edges $e$, summing~(\ref{BKbound}) over all triangles $K$, and noting that $\|\sigma\|_{L^\infty(\Omega)} = \|g_h-\delta\|_{L^\infty(\Omega)} \le C$ and $G_h-G = t (g_h-g)$.
\end{proof}

\begin{lemma}
For every $t \in [0,1]$,
\[
|b_h(G(t);g_h-g,v_h)| \le C \left( h^{-1} \|g_h-g\|_{L^2(\Omega)} + |g_h-g|_{H^1_h(\Omega)} \right) \left( |v_h|_{H^1(\Omega)} + h |v_h|_{H^2_h(\Omega)} \right).
\]
\end{lemma}
\begin{proof}
Let $\sigma = g_h - g$, so that
\[
b_h(G(t);g_h-g,v_h) = \sum_K \langle \sigma, \Hess_G v_h \rangle_{G,K} + \sum_{e \in \mathcal{E}_h} \left\langle \tau_G^T \sigma \tau_G, \left\llbracket \frac{\partial v_h}{\partial n_G} \right\rrbracket \right\rangle_{G,e}.
\]
As was shown in the proof of Lemma~\ref{lemma:bhdiff}, we have
\[
\left\langle \tau_G^T \sigma \tau_G, \left\llbracket \frac{\partial v_h}{\partial n_G} \right\rrbracket \right\rangle_{G,e} = \int_e \tau^T \sigma \tau \llbracket \nabla v_h^T J \widetilde{G} \tau \rrbracket  \, d\ell,
\]
where $\widetilde{G} = G (\det G)^{-1/2} (\tau^T G \tau)^{-1}$.
Thus, if $e = K_1 \cap K_2 \in \mathring{\mathcal{E}}_h$, then the trace inequality gives
\begin{align}
\bigg| \bigg\langle \tau_G^T \sigma \tau_G,  & \left\llbracket \frac{\partial v_h}{\partial n_G} \right\rrbracket \bigg\rangle_{G,e} \bigg| 
\le C\|\tau^T \sigma \tau\|_{L^2(e)} \left( \left\| \left.\nabla v_h\right|_{\partial K_1} \right\|_{L^2(e)} +  \left\| \left.\nabla v_h\right|_{\partial K_2} \right\|_{L^2(e)} \right) \nonumber\\
&\le C \left( h^{-1/2} \|\sigma\|_{L^2(K_1)} + h^{1/2} |\sigma|_{H^1(K_1)}  \right) \nonumber \\
&\;\;\; \times \left( h^{-1/2} |v_h|_{H^1(K_1)} + h^{1/2} |v_h|_{H^2(K_1)} + h^{-1/2} |v_h|_{H^1(K_2)} + h^{1/2} |v_h|_{H^2(K_2)} \right), \nonumber
\end{align}
and similarly for edges $e$ on $\partial \Omega$.
On the other hand, 
\[
|\langle \sigma, \Hess_G v_h \rangle_{G,K} | \le C \|\sigma\|_{L^2(K)} \left( |v_h|_{H^1(K)} + |v_h|_{H^2(K)} \right).
\]
The conclusion follows from summing over all edges $e$ and all triangles $K$.
\end{proof}

We will now finish the proof of Theorem~\ref{thm:err}.  Taking $v_h = P_h v$ and combining Propositions~\ref{prop1},~\ref{prop2}, and~\ref{prop3} with the inverse estimate $|v_h|_{H^2_h(\Omega)} \le C h^{-1} \|v_h\|_{H^1(\Omega)}$, the stability estimate $\|v_h\|_{H^1(\Omega)} \le C \|v\|_{H^1(\Omega)}$, and the upper bound $\log h^{-1} \le h^{-1}$, we obtain
\begin{equation*}\begin{split}
|\langle \kappa_h(g_h) - \kappa(g), v \rangle_g| &\le C \bigg( h^{-1} \|g_h-g\|_{L^2(\Omega)} + |g_h-g|_{H^1_h(\Omega)} + h \inf_{u_h \in V_h} \|\kappa(g) - u_h\|_{L^2(\Omega)} \\
&\hspace{0.1in} + h^{-1}\log h^{-1} \|g_h-g\|_{L^2(\Omega)} \|\kappa_h(g_h)-\kappa(g)\|_{H^{-1}(\Omega,g)} \bigg) \|v\|_{H^1(\Omega)}.
\end{split}\end{equation*}
Taking the supremum over $v \in H^1_0(\Omega)$ and rearranging gives
\begin{equation*}\begin{split}
\big(1-C h^{-1} \log h^{-1} & \|g_h-g\|_{L^2(\Omega)}\big) \|\kappa_h(g_h)-\kappa(g)\|_{H^{-1}(\Omega,g)} \\
&\le C \bigg( h^{-1} \|g_h-g\|_{L^2(\Omega)} + |g_h-g|_{H^1_h(\Omega)} + h \inf_{u_h \in V_h} \|\kappa(g) - u_h\|_{L^2(\Omega)} \bigg).
\end{split}\end{equation*}
Since we have assumed that $\lim_{h \rightarrow 0} h^{-1} \log h^{-1} \|g_h-g\|_{L^2(\Omega)} = 0$, we may divide by $\left(1-C h^{-1} \log h^{-1} \|g_h-g\|_{L^2(\Omega)}\right)$ for $h$ sufficiently small to arrive at the $H^{-1}(\Omega,g)$-estimate in~(\ref{Hm1est}).  

To obtain the $L^2(\Omega)$-estimate in~(\ref{Hm1est}), observe that for any $v \in L^2(\Omega)$, we have
\begin{align*}
\langle & \kappa_h(g_h)  - \kappa(g), v \rangle_g
= \langle \kappa_h(g_h) - \kappa(g), v-P_h v \rangle_g + \langle \kappa_h(g_h) - \kappa(g), P_h v \rangle_g \\
&= \langle P_h \kappa(g) - \kappa(g), v-P_h v \rangle_g + \langle \kappa_h(g_h) - \kappa(g), P_h v \rangle_g \\
&\le \| P_h \kappa(g) - \kappa(g) \|_{L^2(\Omega,g)} \| v-P_h v \|_{L^2(\Omega,g)} + \|\kappa_h(g_h) - \kappa(g)\|_{H^{-1}(\Omega,g)} \|P_h v\|_{H^1(\Omega)} \\
&\le C \left( \| P_h \kappa(g) - \kappa(g) \|_{L^2(\Omega)} \| v \|_{L^2(\Omega,g)} + \|\kappa_h(g_h) - \kappa(g)\|_{H^{-1}(\Omega,g)} \|P_h v\|_{H^1(\Omega)} \right). 
\end{align*}
Now since 
\[
\|P_h v\|_{H^1(\Omega)} \le C h^{-1} \|P_h v\|_{L^2(\Omega)} \le C h^{-1} \|P_h v\|_{L^2(\Omega,g)} \le C h^{-1} \|v\|_{L^2(\Omega,g)},
\]
we deduce that
\begin{align}
\|  \kappa_h(g_h) - \kappa(g) \|_{L^2(\Omega)} 
&\le C \|  \kappa_h(g_h) - \kappa(g) \|_{L^2(\Omega,g)} \nonumber \\
&= C \sup_{v \in L^2(\Omega,g)} \frac{ \langle \kappa_h(g_h) - \kappa(g), v \rangle_g}{\|v\|_{L^2(\Omega,g)}} \nonumber \\
&\le C \left( \inf_{u_h \in V_h} \|\kappa(g)-u_h\|_{L^2(\Omega)} + h^{-1}  \|\kappa_h(g_h) - \kappa(g)\|_{H^{-1}(\Omega,g)} \right). \nonumber
\end{align}
Invoking the $H^{-1}(\Omega,g)$-estimate in~(\ref{Hm1est}) gives
\begin{equation} \label{L2est}
\begin{split}
\|  \kappa_h(g_h) - \kappa(g) \|_{L^2(\Omega)}  \le C \bigg( h^{-2} \|g_h-g\|_{L^2(\Omega)} + h^{-1} |g_h-g|_{H^1_h(\Omega)} \\ + \inf_{u_h \in V_h} \|\kappa(g) - u_h\|_{L^2(\Omega)} \bigg).
\end{split}
\end{equation}
  
Next, assume $\kappa(g) \in H^m(\Omega) \cap H^1_0(\Omega)$, let $0 \le k \le \min\{q+1,m\}$, and let $v_h \in V_h$ be the Scott-Zhang interpolant of $\kappa(g)$.  
Then, using standard inverse estimates and interpolation error estimates, we obtain
\begin{align*}
|  \kappa_h(g_h) & - \kappa(g) |_{H^k_h(\Omega)} 
\le |\kappa_h(g_h) - v_h|_{H^k_h(\Omega)} + |v_h - \kappa(g)|_{H^k_h(\Omega)} \\
&\le Ch^{-k}  \|\kappa_h(g_h) - v_h\|_{L^2(\Omega)} + |v_h - \kappa(g)|_{H^k_h(\Omega)} \\
&\le Ch^{-k} \left( \|\kappa_h(g_h) - \kappa(g)\|_{L^2(\Omega)} +  \|v_h - \kappa(g)\|_{L^2(\Omega)} \right) +  |v_h - \kappa(g)|_{H^k_h(\Omega)}. \\
&\le C \left( h^{-k} \left( \|\kappa_h(g_h) - \kappa(g)\|_{L^2(\Omega)} +  h^\ell|\kappa(g)|_{H^\ell(\Omega)} \right) +  h^{\ell-k} |\kappa(g)|_{H^\ell(\Omega)} \right), 
\end{align*}
for each $\ell=k,k+1,\dots,\min\{q+1,m\}$.
Combining this with the $L^2(\Omega)$-estimate~(\ref{L2est}) gives~(\ref{Hkest}).  This completes the proof of Theorem~\ref{thm:err}.

\section{Numerical Examples} \label{sec:numerical}

In this section, we present numerical experiments to illustrate the convergence rates predicted by Theorem~\ref{thm:err}.  We focus on the $L^2(\Omega)$-error, which, upon taking $k=0$ and $\ell=q+1$ in~(\ref{Hkest}), satisfies
\begin{equation} \label{concreteestimate}
\|\kappa_h(g_h)-\kappa(g)\|_{L^2(\Omega)} 
\le C (h^{r-1} |g|_{H^{r+1}(\Omega)} + h^{q+1} |\kappa(g)|_{H^{q+1}(\Omega)} ),
\end{equation}
assuming that $g \in H^{r+1}(\Omega) \otimes \mathbb{S}$, $\kappa(g) \in H^{q+1}(\Omega)$, and $g_h$ is taken equal to a suitable interpolant of $g$~\cite[Theorem 2.6]{li2018regge}.

\begin{figure} \label{fig:lattice}
\begin{center}
\begin{tikzpicture}[scale=2]
\draw (0,0) 
  -- (0,1) 
  -- (1,0) 
  -- cycle;
\end{tikzpicture}
\begin{tikzpicture}[scale=2]
\draw (0,0) 
  -- (0,1) 
  -- (1,0) 
  -- cycle;
\draw (0.5,0) 
  -- (0.5,0.5)
  -- (0,0.5) 
  -- cycle;
\end{tikzpicture}
\begin{tikzpicture}[scale=2]
\draw (0,0) 
  -- (0,1) 
  -- (1,0) 
  -- cycle;
\draw (0.333,0) 
  -- (0.333,0.333)
  -- (0,0.333) 
  -- cycle;
\draw (0.667,0) 
  -- (0.667,0.333)
  -- (0.333,0.333) 
  -- cycle;
\draw (0.333,0.333) 
  -- (0.333,0.667)
  -- (0,0.667) 
  -- cycle;
\end{tikzpicture}
\end{center}
\caption{Partition of $K$ for $r=0,1,2$.}
\end{figure}
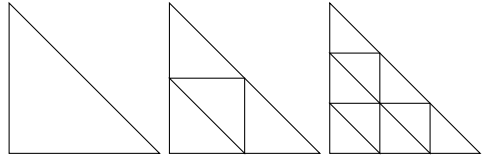

In our numerical implementation, we took $g_h$ equal to a subdivision-based interpolant of $g$ defined as follows.  For each $K \in \mathcal{T}_h$, let $K_1,K_2,\dots,K_{(r+1)^2}$ be the triangles formed by partitioning $K$ along the lines given in barycentric coordinates $(\lambda^0,\lambda^1,\lambda^2)$ by $\lambda^i = j/(r+1)$, $i=0,1,2$, $j=1,2,\dots,r$, as depicted in Figure~\ref{fig:lattice}.  
Let $\mathcal{E}_h^r(K)$ be the union of the edges of $K_1,K_2,\dots,K_{(r+1)^2}$.
Let $\mathcal{E}_h^r = \cup_{K \in \mathcal{T}_h} \mathcal{E}_h^r(K)$.  
For each $e \in \mathcal{E}_h^r$, let $z^{(e)} \in \Omega$ denote the midpoint of $e$.  It is known that the linear functionals
\[
\sigma \mapsto \tau^T \sigma(z^{(e)}) \tau, \quad e \in \mathcal{E}_h^r
\]
form a basis for the dual of $\Sigma_h$, where $\tau$ is a unit vector (relative to the Euclidean metric) tangent to $e$~\cite[pp. 38-39]{li2018regge}.  Let $\{\psi_e\}_{e \in \mathcal{E}_h^r} \subset \Sigma_h$ denote the basis for $\Sigma_h$ dual to these functionals:
\[
\tau^T \psi_e(z^{(e')}) \tau = 
\begin{cases}
1, &\mbox{ if } e = e', \\
0, &\mbox{ otherwise. }
\end{cases}
\]
We took
\begin{equation} \label{interpolant}
g_h = \sum_{e \in \mathcal{E}_h^r} \tau^T g(z^{(e)}) \tau \psi_e.
\end{equation}

\begin{table}[t]
\centering
\begin{filecontents*}{convmore.dat}
   1.6538656e-01   1.5345168e-01   4.4786514e-02   2.8868665e-03
   8.6062109e-02   1.8116900e-01   5.2360620e-02   9.1193713e-04
   4.4423395e-02   1.7867592e-01   5.1516029e-02   3.5230545e-04
   2.2598178e-02   1.7802589e-01   5.1650271e-02   1.5923144e-04
\end{filecontents*}
\pgfplotstabletypeset[
clear infinite,
every head row/.style={before row=\midrule,after row=\midrule},
every last row/.style={after row=\midrule},
    every head row/.append style={
        before row=\toprule,
        before row/.add={}{%
        {} & {}%
        & \multicolumn{4}{c|}{$r=0$} %
        & \multicolumn{4}{c|}{$r=1$} %
        & \multicolumn{4}{c|}{$r=2$}\\ \midrule
        }},
create on use/rate1/.style={create col/gradient loglog={0}{1}},
create on use/rate2/.style={create col/gradient loglog={0}{2}},
create on use/rate3/.style={create col/gradient loglog={0}{3}},
columns={0,1,rate1,2,rate2,3,rate3},
columns/0/.style={dec sep align={|c|},sci,sci 10e,sci zerofill,precision=3,column type/.add={|}{|},column name={$h$}},
columns/1/.style={dec sep align={c|},sci,sci 10e,sci zerofill,precision=3,column type/.add={}{|},column name={Error}},
columns/2/.style={dec sep align={c|},sci,sci 10e,sci zerofill,precision=3,column type/.add={}{|},column name={Error}}, 
columns/3/.style={dec sep align={c|},sci,sci 10e,sci zerofill,precision=3,column type/.add={}{|},column name={Error}}, 
columns/rate1/.style={dec sep align={c|},fixed,zerofill,column type/.add={}{|},column name={Order}},
columns/rate2/.style={dec sep align={c|},fixed,zerofill,column type/.add={}{|},column name={Order}},
columns/rate3/.style={dec sep align={c|},fixed,zerofill,column type/.add={}{|},column name={Order}},
]
{convmore.dat}
\caption{Errors $\|\kappa_h(g_h)-\kappa(g)\|_{L^2(\Omega)}$ for the metric~(\ref{gexact}) on the square $(-1,1) \times (-1,1)$.}
\label{tab:conv}
\end{table}

Table~\ref{tab:conv} shows the errors $\|\kappa_h(g_h)-\kappa(g)\|_{L^2(\Omega)}$ for the metric~(\ref{gexact}) on $\Omega = (-1,1) \times (-1,1)$.  We computed the discrete Gaussian curvature using Definition~\ref{def:kappah} with $q=1$ and $r=0,1,2$ on triangulations with maximum element diameter $h \in [0.02,0.17]$.  We constructed each triangulation by randomly perturbing the interior vertices of a uniform triangulation of $\Omega$, as depicted in Figure~\ref{fig:kappa:a}.  Observe that the $L^2(\Omega)$-error converges linearly when $r=2$ and does not converge when $r<2$, in agreement with the estimate~(\ref{concreteestimate}).

\section*{Acknowledgements}

I am grateful to Melvin Leok for many helpful discussions.

\bibliography{references}

\end{document}